\documentclass[11pt]{amsart}
\usepackage[utf8]{inputenc}
\usepackage{kantlipsum}
\usepackage[pagebackref=true, colorlinks=true,linkcolor=blue, citecolor={red}]{hyperref}
\usepackage{lmodern}\usepackage[T1]{fontenc}
\usepackage{amsmath,amssymb,amsfonts,euscript,graphicx,hyperref,indentfirst}
\usepackage{enumerate}
\usepackage[all]{xy}
\usepackage{tikz-cd}
\usepackage{stmaryrd}
\usepackage{relsize,exscale}
\calclayout

\newcommand{\hcoker}{/\!\!/}

\newtheorem{thmx}{{\bf Theorem}}[section]

\newtheorem*{corollary*}{Corollary}

\newtheorem{theorem}{Theorem}[section]
\newtheorem{lemma}[theorem]{Lemma}
\newtheorem{proposition}[theorem]{Proposition}
\newtheorem{corollary}[theorem]{Corollary}

\newtheorem{definition-theorem}[theorem]{Definition-Theorem}

\theoremstyle{definition}
\newtheorem{definition}[theorem]{Definition}
\newtheorem{example}[theorem]{Example}

\theoremstyle{remark}
\newtheorem{remark}[theorem]{Remark}

\makeatletter
\def\l@subsection{\@tocline{2}{0pt}{2.5pc}{5pc}{}} 
\makeatother

\numberwithin{equation}{section}

\title{Topological aspects of the dynamical moduli space of rational maps}
\author{Maxime Bergeron, Khashayar Filom and Sam Nariman}

\address{Maxime Bergeron, Department of Mathematics,
  University of Chicago;
 5734 S. University Avenue, Chicago IL, 60637,
USA}
\email{mbergeron@math.uchicago.edu}

\address{Khashayar Filom, Department of Mathematics,
  University of Michigan;
 530 Church St, Ann Arbor, MI 48109,
USA}
\email{filom@umich.edu}

\address{Sam Nariman, Department of Mathematics,
   Purdue University;
150 N. University Street;
West Lafayette, IN 47907-2067}
\email{snariman@purdue.edu}

\begin{document}

\begin{abstract}
    We investigate the topology of the space of M\"obius conjugacy classes of degree $d$ rational maps on the Riemann sphere. We show that it is rationally acyclic and we compute its fundamental group. As a byproduct, we also obtain the ranks of some higher homotopy groups of the parameter space of degree $d$ rational maps allowing us to extend the previously known range. Moreover, we show that this parameter space is not nilpotent. 
\end{abstract}

\maketitle
\tableofcontents

\section{Introduction}\label{intro}
The dynamical moduli space of rational maps is defined as the space of M\"obius conjugacy classes of rational maps from $\mathbb{CP}^1$ to $\mathbb{CP}^1$.
 To keep things concrete, we can identify a given rational map of degree $d$ with a pair of homogeneous polynomials of degree $d$ in two variables without common zeroes. The space of such rational maps ${\rm{Rat}}_d$ is then endowed with an action of the group of M\"obius transformations ${\rm{PSL}}_2(\mathbb{C})$ by conjugation $$
\alpha\cdot f:=\alpha\circ f\circ\alpha^{-1}.
$$
 As a complex orbifold, the moduli space $\mathcal{M}_d$ of degree $d$ rational maps from $\mathbb{CP}^1$ to $\mathbb{CP}^1$ is then constructed as the quotient  
 \begin{equation}\label{definition}
\mathcal{M}_d={\rm{Rat}}_d\big/{\rm{PSL}}_2(\Bbb{C}).
\end{equation}

Aside from their ubiquity in dynamics, spaces such as ${\rm{Rat}}_d$ are intriguing because they lie at the heart of a growing body of work on  interactions between topology, algebraic geometry and number theory. 
In this context, ${\rm{Rat}}_d$ is also one of the main examples in the work of Farb and Wolfson \cite{MR3548124,MR3652084} on spaces of polynomials generalizing the classical study of discriminants and resultants. 
Moreover, in the topology literature, the space ${\rm{Rat}}_d$ is studied in the seminal paper \cite{MR533892} of Segal. 
Nevertheless, in most cases, the topology of the moduli spaces \eqref{definition} which is a quotient of ${\rm{Rat}}_d$ has remained a mystery.

Accordingly, our goal in this paper is to study the topology of $\mathcal{M}_d$. In the case of quadratic rational maps, the features of $\mathcal{M}_2$ have previously been extensively studied by Milnor \cite{MR1246482}. In particular,  he identified $\mathcal{M}_2$ with  $\mathbb{C}^2$ through dynamical methods. This identification was then generalized by Silverman \cite{MR1635900} who used  Geometric Invariant Theory to construct $\mathcal{M}_d$ as an integral affine scheme over $\mathbb{Z}$ and identify $\mathcal{M}_2$ with the affine plane $\mathbb{A}^2$.  Our first two main theorems can be viewed as topological extensions of these results to higher degrees:
 \begin{thmx}\label{main 1}
If the integer $d$ is greater than $1$, then  $\mathcal{M}_d$ is rationally acyclic.
\end{thmx}
\begin{thmx}
		If the integer $d$ is greater than $1$, then \label{main 2}
		\begin{equation*}
\pi_1(\mathcal{M}_d) = \begin{cases}
0 &\text{if $d$ is even, and}\\
\Bbb{Z}/2\Bbb{Z} &\text{if $d$ is odd.}
\end{cases}
\end{equation*}
\end{thmx}
In particular, using the rational Hurewicz theorem, we deduce that $\mathcal{M}_d$ is rationally weakly contractible if $d$ is even.

It should be remarked at this point that the moduli space $\mathcal{M}_d$ was shown to be rational as a variety by Levy \cite{MR2741188} and that the point counts of $\mathcal{M}_d$ and $\mathbb{A}^{2d-2}$ are  known to agree over all finite fields by work of Gunther and West \cite{MRGuntherWestUnpublished}. From this perspective, our results enhance the similarities between the topological invariants of $\mathcal{M}_d$ over the rational numbers and those of affine space.
 In particular, one might be tempted to conjecture that $\mathcal{M}_d$ should be contractible over the complex numbers whenever $d$ is even. 
 Nevertheless, using cohomology with finite coefficients, we shall see that this is not the case (cf. Theorem \ref{main finite} and Corollary \ref{not contractible}):
  \begin{thmx}\label{main 2.5}
If the integer $d$ is greater than or equal to $7$, then  $\mathcal{M}_d$ is not contractible.
\end{thmx}

To prove these theorems, we rely on dynamical results of  Milnor \cite{MR1246482}, Silverman \cite{MR1635900} and Miasnikov, Stout and Williams \cite{MR3709645}  as well as more topological tools such as the \emph{generalized} Leray-Hirsch theorem developed by Peters and Steenbrink \cite{MR2078574} and a characterization of the space of Toeplitz matrices by Milgram \cite{MR1445557}. We are also inspired by the rich history in homotopy theory of the parameter space ${\rm{Rat}}^*_d$ of based rational maps of degree $d$. Notably, recall that its stable homotopy type was computed by F. Cohen, R. Cohen, Mann and Milgram \cite{MR1097023} following up on work of Segal \cite{MR533892} showing  that, for $1\leq i<d$, the $i^{\rm{th}}$ homotopy group of ${\rm{Rat}}^*_d$ coincides with the $i^{\rm{th}}$ homotopy group of the space of based continuous self-maps of $S^2$ of the same degree $d$. Within this range, it turns out that all higher homotopy groups of ${\rm{Rat}}^*_d$  are torsion. On the other hand, Guest, Kozlowski, Murayama and Yamaguchi \cite{MR1365252} showed that if $1\leq i<d$, then the group $\pi_i({\rm{Rat}}_d)$ is infinite only when $i=3$ in which case $\pi_3({\rm{Rat}}_d)$ is of rank one. 
Taking the quotient of ${\rm{Rat}}_d$ by the action of ${\rm{PSL}}_2(\Bbb{C})$ kills this non-torsion part and an application of the rational Hurewicz theorem to a  space related to 
$\mathcal{M}_d={\rm{Rat}}_d\big/{\rm{PSL}}_2(\Bbb{C})$ allows us to obtain the rank of higher homotopy groups of ${\rm{Rat}}_d$ and ${\rm{Rat}}^*_d$ beyond the previously known range:

\begin{thmx}\label{main 3}
Let $\phi$ denote Euler's totient function. Given $d>3$ and $d\leq i\leq 2d-2$, one has 
$$
{\rm{rank}}\,\pi_i({\rm{Rat}}^*_d)={\rm{rank}}\,\pi_i({\rm{Rat}}_d)=
\begin{cases}
\phi\left(\frac{d}{j}\right)&\text{ if } i=2d-2j \text{ where }1\leq j<d \text{ divides } d\text{, and}\\
0 & \text{otherwise.}\\
\end{cases}
$$
 
\end{thmx}

To put this theorem into more context, recall that Segal showed that ${\rm{Rat}}_d$ and ${\rm{Rat}}^*_d$ are stably nilpotent spaces \cite[Corollary 6.3]{MR533892}. On the other hand, using our theorem, we show that the fundamental groups of  ${\rm{Rat}}_d$ and ${\rm{Rat}}^*_d$ do not act nilpotently on the corresponding homotopy groups in dimension $2d-2$, so Segal's result does not hold unstably.

\begin{thmx}\label{not nilpotent}
The spaces ${\rm{Rat}}^*_d$ and ${\rm{Rat}}_d$ are not nilpotent for any $d>3$.
\end{thmx}



We conclude this introduction with a brief outline of the paper. We begin in  \S\ref{background} with background material on the spaces ${\rm{Rat}}_d$ and $\mathcal{M}_d$. This is followed by a proof of Theorem \ref{main 1} in  \S\ref{rational homology}. The difficulty in understanding the homology groups of $\mathcal{M}_d={\rm{Rat}}_d\big/{\rm{PSL}}_2(\Bbb{C})$ from those of ${\rm{PSL}}_2(\Bbb{C})$ and ${\rm{Rat}}_d$ arises due to the fact that the conjugation action of ${\rm{PSL}}_2(\Bbb{C})$ on ${\rm{Rat}}_d$ is not free. We circumvent this with an application of the generalized Leray-Hirsch theorem. Similar arguments in \S\ref{application marked} and \S\ref{application variant} allow us to understand the rational homology of variants of the moduli space $\mathcal{M}_d$ such as the orbit space $\mathcal{M}^{\rm{post}}_d$ for the post-composition action  of ${\rm{PSL}}_2(\Bbb{C})$ on ${\rm{Rat}}_d$  rather than the conjugation action; see  Theorem \ref{main variant}. This leads us to the proof of Theorem \ref{main 2.5} in \S\ref{finite homology}. The key point here is that, unlike the rational homology groups, the homology groups of $\mathcal{M}_d$ with finite coefficients carry a rich structure. Indeed, there is an algebra isomorphism between $H^*(\mathcal{M}_d;\Bbb{F}_p)$ and $H^*(\mathcal{M}^{\rm{post}}_d;\Bbb{F}_p)$ for all but finitely many characteristics $p$, see  Theorem \ref{main finite}, and the latter cohomology ring is highly non-trivial.
The proof of Theorem \ref{main 2} is found in \S\ref{homotopy groups}. Once again, the difficulty lies in the fact that the action of ${\rm{PSL}}_2(\Bbb{C})$ on ${\rm{Rat}}_d$ is not free. Nevertheless, it turns out to be proper; see Lemma \ref{proper}. Finally, Theorems \ref{main 3} and \ref{not nilpotent} are proved in \S\ref{application resultant}.

\textbf{Acknowledgements.}

The first named author would like to thank Ronno Das, Benson Farb, Claudio G\'omez-Gonz\'ales and, more generally, the UChicago geometry \& topology research group for inspiring conversations. He was partially supported by the Jump Trading Mathlab Research Fund, an NSERC postdoctoral fellowship and NSF award  DMS - 1704692. The second named author is grateful to Laura DeMarco, Benson Farb, and Jesse Wolfson for helpful comments and conversations. The third named author thanks Wolfgang L{\"u}ck for the reference \cite{MR1027600}. He  was partially supported by NSF Grant DMS - 1810644. The authors would like to thank the organizers of the NSF/PIMS summer school -- The Roots of Topology, the IPAM workshop on Braids, Resolvent Degree and Hilbert’s 13th Problem and the PIMS workshop on Arithmetic Topology where some of this work was conducted.
Finally, the authors are grateful to the anonymous referee for his/her valuable comments and suggestions.

\section{Background on the spaces \texorpdfstring{${\rm{Rat}}_d$}{Ratd} and \texorpdfstring{$\mathcal{M}_d$}{Md}}\label{background}

In this section, we recall results that we shall need concerning the parameter and moduli spaces of rational maps of a fixed degree $d$. We will assume $d$ is greater than $1$ throughout the paper. 

\subsection{The parameter space \texorpdfstring{${\rm{Rat}}_d$}{Ratd}}\label{background1}
We denote the space of degree $d$ rational maps (holomorphic maps $\Bbb{CP}^1\rightarrow\Bbb{CP}^1$ of degree $d$) by ${\rm{Rat}}_d$. This is the complement of the \textit{resultant hypersurface} in the projective space $\Bbb{P}\left(\Bbb{C}^{d+1}\times\Bbb{C}^{d+1}\right)=\Bbb{CP}^{2d+1}$ of pairs of polynomials of degree at most $d$. As such, ${\rm{Rat}}_d$ is an affine variety of dimension $2d+1$.  

A celebrated result of Segal states that the inclusion of ${\rm{Rat}}_d$ into the space ${\rm{Map}}_d(S^2)$ of continuous degree $d$ maps $S^2\rightarrow S^2$ induces isomorphisms on homotopy groups $\pi_i$ for $0\leq i<d$ and an epimorphism on $\pi_d$ \cite[Proposition 1.1$'$]{MR533892}. 
Within this range, the homotopy groups of ${\rm{Rat}}_d$ have been computed:

\begin{theorem}[{\cite[Theorem 1]{MR1365252}}]\label{Rat homotopy}
Assuming $d\geq 3$, one has
$$
\pi_i({\rm{Rat}}_d)\cong
\begin{cases}
\Bbb{Z}/2\Bbb{Z} & d\geq 3, i=2\\
\pi_{i+2}(S^2)\oplus\pi_{i}(S^3)& 3\leq i<d.
\end{cases}
$$
\end{theorem}

It should be noted that, rather than working with ${\rm{Rat}}_d$, many authors 
study the parameter space ${\rm{Rat}}_d^*$ of holomorphic maps preserving a base point. For instance, there is a description of the stable homotopy type of ${\rm{Rat}}_d^*$ given in  \cite[Theorem 1.1]{MR1097023} which is used to understand its cohomology groups in terms of those of an Artin braid group. 
For our purposes, we will often identify ${\rm{Rat}}_d^*$ with the space of those rational maps of degree $d$ that send $\infty$ to $1$. This is the approach taken in \cite{MR3548124} where the authors compute the Betti numbers for the space of holomorphic maps $\Bbb{CP}^1\rightarrow\Bbb{CP}^{m-1}$ of degree $d$ taking $\infty$ to $[1:\dots:1]$; cf. \cite[Example 6.12]{MR3720801}. Nevertheless, preservation of base points is a dynamically restrictive hypothesis so we turn our attention to the parameter space ${\rm{Rat}}_d$. In the proposition below, we  use the results of \cite{MR3548124} to compute the Betti numbers of ${\rm{Rat}}_d$ itself:
\begin{proposition}\label{Rat homology}
For any $d\geq 2$ one has
$$b_i({\rm{Rat}}_d)=\begin{cases}
1\quad i=0\text{ or } 3\text{, and}\\
0\quad \text{otherwise.}
\end{cases}$$
\end{proposition}
\begin{proof}
In positive dimensions, the only non-zero Betti number of ${\rm{Rat}}^*_d$ is known to be $b_1({\rm{Rat}}^*_d)=1$ \cite[Theorem 1.2(2)]{MR3548124}. Picking an arbitrary point $z_0$ of $\Bbb{CP}^1$, the evaluation map $f\mapsto f(z_0)$ results in a fiber bundle: 
\begin{equation}\label{fibration1}
{\rm{Rat}}_d^*\hookrightarrow{\rm{Rat}}_d\rightarrow\Bbb{CP}^1.    
\end{equation}
On the $E_2$ page of the corresponding Serre spectral sequence, the terms 
$$
E_2^{i,j}=H^i\left(\Bbb{CP}^1;H^j({\rm{Rat}}_d^*;\Bbb{Q})\right)
$$
are non-zero precisely when $(i,j)\in\{(0,0),(0,1),(2,0),(2,1)\}$ in which case they are all one-dimensional. 
On the $E_3$ page one has 
$$
E_3^{0,0}\cong E_2^{0,0}\cong\Bbb{Q},\quad E_3^{2,1}\cong E_2^{2,1}\cong\Bbb{Q};
$$
and
$$
E_3^{0,1}\cong{\rm{ker}}\left(d_2:E_2^{0,1}\rightarrow E_2^{2,0}\right),\quad E_3^{2,0}\cong{\rm{coker}}\left(d_2:E_2^{0,1}\rightarrow E_2^{2,0}\right);
$$
with all other terms zero. We conclude that the spectral sequence degenerates at the $E_3$ page due to the fact that $E_3^{i,j}=0$ whenever $i\geq 3$. Consequently, the only potentially non-trivial rational cohomology groups of ${\rm{Rat}}_d$ are: 
\begin{equation*}
\begin{split}
&H^0({\rm{Rat}}_d;\Bbb{Q})\cong E_3^{0,0}\cong\Bbb{Q},\\
&H^1({\rm{Rat}}_d;\Bbb{Q})\cong E_3^{0,1}\cong{\rm{ker}}\left(d_2:E_2^{0,1}\rightarrow E_2^{2,0}\right),\\
&H^2({\rm{Rat}}_d;\Bbb{Q})\cong E_3^{2,0}\cong{\rm{coker}}\left(d_2:E_2^{0,1}\rightarrow E_2^{2,0}\right),\\
&H^3({\rm{Rat}}_d;\Bbb{Q})\cong E_3^{2,1}\cong\Bbb{Q}.
\end{split}    
\end{equation*}
However, $H^1({\rm{Rat}}_d;\Bbb{Q})$ must be trivial since $\pi_1({\rm{Rat}}_d)$ is known to be a finite group; see Proposition \ref{Rat fundamental} below. This implies that the differential $d_2:E_2^{0,1}\rightarrow E_2^{2,0}$ is injective and thus, given the fact that its domain and codomain are one-dimensional, an isomorphism.  The groups 
$H^1({\rm{Rat}}_d;\Bbb{Q})$ and $H^2({\rm{Rat}}_d;\Bbb{Q})$ above are therefore both trivial. 
\end{proof}
\begin{remark}
These Betti numbers can also be obtained from the work of Claudio Gómez-Gonzáles on the  rational cohomology of the space of degree $d$ holomorphic maps from 
$\mathbb{CP}^n$ to $\mathbb{CP}^m$ \cite{MR4140094}; cf. Remark \ref{claudio}.
\end{remark}

We conclude this subsection with a discussion of the fundamental group of ${\rm{Rat}}_d$, adapted from \cite[Appendix B]{MR1246482}. Taking ${\rm{Rat}}_d^*$ to be the subspace of degree $d$ rational maps under which $\infty\mapsto 1$, any element $f$ of this space may be uniquely written as 
$f(z)=\frac{p(z)}{q(z)}$ where $p$ and $q$ are coprime monic polynomials of degree $d$. It is known that the resultant map
\begin{equation}\label{resultant map}
\begin{cases}
{\rm{Rat}}_d^*\rightarrow\Bbb{C}-\{0\}\\
f=\frac{p}{q}\mapsto{\rm{Res}}(p,q)    
\end{cases}
\end{equation}
is a fibration that induces an isomorphism of fundamental groups \cite[Propositions 6.2 and 6.4]{MR533892}. The fiber above $1$ is the \textit{resultant=1 hypersurface} 
\begin{equation}\label{resultant=1}
\mathcal{R}_d=\left\{(p,q)\,\big|\,p,q\in\Bbb{C}[z]\text{ monic, coprime and of degree }d \text{ with }{\rm{Res}}(p,q)=1\right\}    
\end{equation}
that will come up again in \S\ref{application resultant}. Now, writing the long exact sequence of homotopy groups for the fibration 
\begin{equation}\label{fibration2}
\mathcal{R}_d\hookrightarrow{\rm{Rat}}^*_d\rightarrow\Bbb{C}-\{0\}    
\end{equation}
implies that $\mathcal{R}_d$ is simply connected. This helps us to construct a universal covering space for ${\rm{Rat}}_d$. To this end, consider  
\begin{equation}
    \widetilde{\rm{Rat}}_d:=\left\{(p,q)\,\big|\,p,q\in\Bbb{C}[z]\text{ coprime with }\max\{\deg p, \deg q\}=d\text{ and }{\rm{Res}}(p,q)=1\right\}
\end{equation}
and observe that $\mathcal{R}_d$ also fits into a fiber bundle
\begin{equation}\label{fibration3}
\mathcal{R}_d\hookrightarrow\widetilde{\rm{Rat}}_d
\rightarrow\Bbb{C}^2-\{(0,0)\}    
\end{equation}
\normalsize
where the projection onto $\Bbb{C}^2-\{(0,0)\}$ maps $(p,q)$ to the ordered pair consisting of the coefficients of $z^d$ in $p$ and $q$. As the fiber and the base are both simply connected, we deduce that $\widetilde{\rm{Rat}}_d$ is simply connected as well. There is a surjective map
\begin{equation}\label{universal cover}
\begin{cases}
\widetilde{\rm{Rat}}_d\rightarrow {\rm{Rat}}_d\\
(p,q)\mapsto f=\frac{p}{q}
\end{cases}    
\end{equation}
whose fibers are scalar multiples $(\lambda p,\lambda q)$ of $(p,q)$ satisfying ${\rm{Res}}(\lambda p,\lambda q)=1$. The left hand side is simply $\lambda^{2d}{\rm{Res}}(p,q)$ and $\lambda$  must thus be a $2d^{\rm{th}}$ root of unity. We conclude that \eqref{universal cover} is the quotient map for the scaling action 
\begin{equation}\label{action1}
\lambda\cdot(p,q):=(\lambda p,\lambda q)
\end{equation}
of the group of $2d^{\rm{th}}$ roots of unity on $\widetilde{\rm{Rat}}_d$. The action is free and $\widetilde{\rm{Rat}}_d$ is simply connected. Consequently, \eqref{universal cover} is a universal covering map. To summarize:

\begin{proposition}\label{Rat fundamental}
The fundamental group of ${\rm{Rat}}_d^*$ can be identified with $\Bbb{Z}$ via the map \eqref{resultant map}. The fundamental group of 
${\rm{Rat}}_d$ is cyclic of order $2d$ and the homomorphism $\pi_1({\rm{Rat}}_d^*)\rightarrow\pi_1({\rm{Rat}}_d)$ induced by inclusion may be identified with 
$\Bbb{Z}\twoheadrightarrow\Bbb{Z}/2d\Bbb{Z}$. Moreover, the universal cover of ${\rm{Rat}}_d$ is given by \eqref{universal cover}. 
\end{proposition}

\begin{remark}
The computation of the fundamental groups of ${\rm{Rat}}_d^*$ and ${\rm{Rat}}_d$ above is adapted from \cite[Appendix B]{MR1246482} but also appears in earlier papers. 
A proof that $\pi_1({\rm{Rat}}_d^*)\cong\Bbb{Z}$ is presented in Segal's work (where it is attributed to J. D. S. Jones). The isomorphism 
$\pi_1({\rm{Rat}}_d)\cong\Bbb{Z}/2d\Bbb{Z}$ can also be obtained by combining the earlier results 
\cite[Proposition 1.1$'$]{MR533892} and \cite[Proposition 3.3]{MR744651}.
\end{remark}

\begin{remark}
Rational maps with some prescribed information on their critical points or ramification are studied in the literature too. Paper \cite{MR1093002} is concerned with rational maps with a given set of critical points; and paper \cite{MR4031111} computes the cohomology of spaces of polynomials with a bound on their \textit{ramification length}.
In both papers, the maps should be considered modulo post-composition by M\"obius/affine transformations. The post-composition action will also come up later in \S\ref{application variant}. 
\end{remark}

\subsection{The moduli space \texorpdfstring{$\mathcal{M}_d$}{Md}}\label{background2}
The group of automorphisms of the Riemann sphere
$$
\left\{\alpha:z\mapsto\frac{rz+s}{tz+u}\,\bigg|\,
\begin{bmatrix}
r&s\\
t&u
\end{bmatrix}
\in{\rm{SL}}_2(\Bbb{C})\right\}\cong {\rm{PSL}}_2(\Bbb{C})
$$
 acts on ${\rm{Rat}}_d$ from the left 
\begin{equation}\label{auxiliary1}
\alpha\cdot f(z)=\alpha\left(f\left(\alpha^{-1}(z)\right)\right)=\frac{rf\left(\frac{uz-s}{-tz+r}\right)+s}{tf\left(\frac{uz-s}{-tz+r}\right)+u}.    
\end{equation}
\begin{remark}\label{action on cover}
Since this will be used later, we bring the reader's attention to the fact that this action lifts to an action of ${\rm{SL}}_2(\Bbb{C})$ on the universal cover $\widetilde{\rm{Rat}}_d$ that appeared in \eqref{fibration3}. Homogenizing an element $(p(z),q(z))$ of $\widetilde{\rm{Rat}}_d$ as 
$\left(P(X,Y),Q(X,Y)\right)$ where 
$$
P(X,Y)=Y^d\, p\left(\frac{X}{Y}\right),\quad\text{and } Q(X,Y)=Y^d\, q\left(\frac{X}{Y}\right),
$$
the action ${\rm{SL}}_2(\Bbb{C})\curvearrowright\widetilde{\rm{Rat}}_d$ is given by

\begin{equation}\label{action2}
\begin{bmatrix}
r&s\\
t&u
\end{bmatrix}\cdot
\left(P(X,Y),Q(X,Y)\right)=(P'(X,Y),Q'(X,Y))\end{equation}
where
\begin{equation*}
    P'(X,Y):=rP(uX-sY,-tX+rY)+sQ(uX-sY,-tX+rY)
\end{equation*}
and
\begin{equation*}
    Q'(X,Y):=tP(uX-sY,-tX+rY)+uQ(uX-sY,-tX+rY).
\end{equation*}
This action  is well defined since transforming a pair $(P,Q)$ by a square matrix via \eqref{action2}
multiplies the resultant by a power of the determinant \cite[Exercise 2.7]{MR2316407}. Notice that \eqref{action2} also provides a homogenization of the  conjugation action of ${\rm{SL}}_2(\Bbb{C})$ (or ${\rm{PSL}}_2(\Bbb{C})$) on the space $\Bbb{CP}^{2d+1}\supset{\rm{Rat}}_d$ of projective equivalence classes of pairs of polynomials of degree at most $d$.
\end{remark}

The group ${\rm{Aut}}(f)$ of M\"obius transformations commuting with a rational map $f$ of degree $d\geq 2$ is finite as any such transformation must preserve a subset of cardinality at least three formed by critical and fixed points of $f$. Therefore, the quotient $\mathcal{M}_d$ of ${\rm{Rat}}_d$ under this action is a complex orbifold of dimension 
$$
(2d+1)-3=2d-2.
$$
In particular, the set of orbifold points is the locus determined by rational maps admitting non-trivial M\"obius automorphisms. 
\begin{definition}\label{symmetry locus definition}
The \textit{symmetry locus} in $\mathcal{M}_d$ is defined as:
\begin{equation}\label{S}
\mathcal{S}:=\left\{\langle f\rangle\in\mathcal{M}_d\mid f\in{\rm{Rat}}_d, {\rm{Aut}}(f)\neq\{1\}\right\}.
\end{equation}
\end{definition}
As mentioned in the introduction, $\mathcal{M}_d$ also has the structure of an affine variety. The singular locus of this variety is closely related to the orbifold locus.
\begin{proposition}[{\cite[Corollary 2.10 and Theorem 5.8]{MR3709645}}]\label{symmetry locus}
The symmetry locus $\mathcal{S}$ is a closed subvariety of dimension (codimension) $d-1$. For $d\geq 3$ it coincides with the singular locus of $\mathcal{M}_d$. 
\end{proposition}
\noindent 
Non-free orbits of the conjugation action ${\rm{PSL}}_2(\Bbb{C})\curvearrowright{\rm{Rat}}_d$ will arise again in  
\S\ref{finite homology} and \S\ref{homotopy groups}.

\section{The homology of \texorpdfstring{$\mathcal{M}_d$}{Md}}\label{homology}
Building on results about ${\rm{Rat}}_d$ outlined above, we turn our attention to the homology of $\mathcal{M}_d$.

\subsection{The rational homology of \texorpdfstring{$\mathcal{M}_d$}{Md}}\label{rational homology}
Recall from \S\ref{background2} that the action of ${\rm{PSL}}_2(\Bbb{C})$ on ${\rm{Rat}}_d$ is not free and, as such, that  
\begin{equation}\label{fibration5}
{\rm{PSL}}_2(\Bbb{C})\rightarrow{\rm{Rat}}_d\xrightarrow{{f\mapsto\langle f\rangle}}\mathcal{M}_d    
\end{equation}
is not a fiber bundle. Nevertheless, it is an ``orbifold fiber bundle'' in the sense of \cite{MR2078574} because it is the geometric quotient for an action of a connected group with finite stabilizers. 
Consequently, a generalization of the classical Leray-Hirsch theorem (see \cite[Theorem 2]{MR2078574}) yields an isomorphism   
\begin{equation}\label{isomorphism}
H^*({\rm{Rat}}_d;\Bbb{Q})\cong H^*(\mathcal{M}_d;\Bbb{Q})\otimes H^*({\rm{PSL}}_2(\Bbb{C});\Bbb{Q})     
\end{equation}
in this orbifold context provided that the inclusion of an orbit (and hence every orbit)
\begin{equation}\label{fiber inclusion}
\begin{cases}
{\rm{PSL}}_2(\Bbb{C})\rightarrow{\rm{Rat}}_d\\
\alpha\mapsto \alpha\circ f\circ\alpha^{-1}  
\end{cases}    
\end{equation}
induces a surjection $H^*({\rm{Rat}}_d;\Bbb{Q})\twoheadrightarrow H^*({\rm{PSL}}_2(\Bbb{C});\Bbb{Q})$. Our strategy for proving  \eqref{isomorphism} is to establish this surjectivity. Notice that the rational cohomology groups involved are very sparse: $H^*({\rm{Rat}}_d;\Bbb{Q})$ is non-trivial only in dimensions zero and three (Proposition \ref{Rat homology}) and the same holds for the Lie group ${\rm{PSL}}_2(\Bbb{C}).$ 
Consequently, in order to apply the generalized Leray-Hirsch theorem, it suffices to establish the following:

\begin{lemma}\label{main lemma}
The conjugation orbit inclusion \eqref{fiber inclusion} induces an isomorphism 
\begin{equation}\label{main isomorphism}
H_3({\rm{PSL}}_2(\Bbb{C});\Bbb{Q})\stackrel{\sim}{\rightarrow}H_3({\rm{Rat}}_d;\Bbb{Q})    
\end{equation}
of one-dimensional $\Bbb{Q}$-vector spaces.
\end{lemma}

Assuming the lemma, we now prove our first main result:
\begin{proof}[Proof of Theorem \ref{main 1}]
By the preceding discussion,  Lemma \ref{main lemma} implies that the inclusion \eqref{fiber inclusion} induces an isomorphism 
$H^*({\rm{Rat}}_d;\Bbb{Q})\stackrel{\sim}{\rightarrow} H^*({\rm{PSL}}_2(\Bbb{C});\Bbb{Q})$. Invoking the generalized Leray-Hirsch theorem (\cite[Theorem 2]{MR2078574}), we obtain   isomorphism \eqref{isomorphism}. The graded vector spaces $H^*({\rm{PSL}}_2(\Bbb{C});\Bbb{Q})$ and $H^*({\rm{Rat}}_d;\Bbb{Q})$ are then isomorphic and  $H^*(\mathcal{M}_d;\Bbb{Q})$ is non-trivial only in degree $0$. 
\end{proof}

\begin{proof}[Proof of Lemma \ref{main lemma}]
The action of ${\rm{PSL}}_2(\Bbb{C})$ on ${\rm{Rat}}_d$ lifts to the action of ${\rm{SL}}_2(\Bbb{C})$ on the universal cover $\widetilde{\rm{Rat}}_d$ mentioned in Remark \ref{action on cover}.
Supposing for now that $d>3$, by Theorem \ref{Rat homotopy} we have 
\begin{equation}\label{auxiliary2}
\begin{split}
&\pi_2(\widetilde{\rm{Rat}}_d)\cong\pi_2({\rm{Rat}}_d)\cong\pi_4(S^2)\oplus\pi_2(S^3)\cong\Bbb{Z}/2\Bbb{Z}, and\\
&\pi_3(\widetilde{\rm{Rat}}_d)\cong\pi_3({\rm{Rat}}_d)\cong\pi_5(S^2)\oplus\pi_3(S^3)\cong\Bbb{Z}/2\Bbb{Z}\oplus\Bbb{Z}.
\end{split}
\end{equation}
As such, $\widetilde{\rm{Rat}}_d$ is a simply connected space with $\pi_2(\widetilde{\rm{Rat}}_d)\otimes\Bbb{Q}=0$. The rational Hurewicz theorem now implies that 
$\Bbb{Q}\cong\pi_3(\widetilde{\rm{Rat}}_d)\otimes\Bbb{Q}\rightarrow H_3(\widetilde{\rm{Rat}}_d;\Bbb{Q})$ 
is an isomorphism. This yields the commutative diagram 
\begin{equation}\label{diagram}
\xymatrixcolsep{5pc}\xymatrix{\pi_3({\rm{SL}}_2(\Bbb{C}))\otimes\Bbb{Q} \ar[d]^{\sim}_{\text{rational Hurewicz}}\ar[r] 
& \pi_3(\widetilde{\rm{Rat}}_d)\otimes\Bbb{Q} \ar[d]^{\text{rational Hurewicz}}_{\sim}\\
H_3({\rm{SL}}_2(\Bbb{C});\Bbb{Q}) \ar[d]^{\sim} \ar[r]& H_3(\widetilde{\rm{Rat}}_d;\Bbb{Q}) \ar@{->>}[d]\\
H_3({\rm{PSL}}_2(\Bbb{C});\Bbb{Q})\ar[r]& H_3({\rm{Rat}}_d;\Bbb{Q})\ar@/_/[u]_{\text{transfer}}.
}
\end{equation} 
In the right column, the first and the last vector spaces $\pi_3(\widetilde{\rm{Rat}}_d)\otimes\Bbb{Q}$ and $H_3({\rm{Rat}}_d;\Bbb{Q})$ are one-dimensional (cf. \eqref{auxiliary2} and Proposition \ref{Rat homology}).  In the same column, the first arrow is an isomorphism due to the rational Hurewicz theorem, and the second one is surjective as it has a section given by the transfer homomorphism.  We deduce that in diagram \eqref{diagram} all vector spaces are one-dimensional and all vertical arrows are isomorphisms. Hence, to establish the isomorphism \eqref{main isomorphism}, it suffices to show that the following linear transformation of one-dimensional $\Bbb{Q}$-vector spaces
$$
\pi_3({\rm{PSL}}_2(\Bbb{C}))\otimes\Bbb{Q}\cong\pi_3({\rm{SL}}_2(\Bbb{C}))\otimes\Bbb{Q}\rightarrow\pi_3(\widetilde{\rm{Rat}}_d)\otimes\Bbb{Q}\cong\pi_3({\rm{Rat}}_d)\otimes\Bbb{Q}
$$
induced by the inclusion of an orbit for the conjugation action of ${\rm{PSL}}_2(\Bbb{C})$ on ${\rm{Rat}}_d$ is an isomorphism.\\
\indent
Fix an arbitrary $f_0$ and consider the corresponding orbit.  Post-composing the last homomorphism with the map 
$$\pi_3({\rm{Rat}}_d)\otimes\Bbb{Q}\rightarrow\pi_3(\Bbb{CP}^1)\otimes\Bbb{Q}\cong\Bbb{Q}$$
induced by evaluation  
$f\mapsto f(\infty)$, we observe that it suffices to show that 
\begin{equation}\label{auxiliary3}
\begin{cases}
{\rm{PSL}}_2(\Bbb{C})\rightarrow\Bbb{CP}^1\\
\alpha
\mapsto 
\alpha\circ f_0\circ\alpha^{-1}\big|_{z=\infty}
\end{cases}    
\end{equation}
induces an isomorphism on the third rational homotopy groups. Notice here that ${\rm{Rat}}_d$ is connected and hence the homotopy class of \eqref{auxiliary3} is independent of our choice of $f_0$. We can now replace ${\rm{PSL}}_2(\Bbb{C})$ with its maximal compact subgroup ${\rm{PSU}}_2$ to which ${\rm{PSL}}_2(\Bbb{C})$ deformation retracts. The resulting map may be written as the composition of 
\begin{equation}\label{auxiliary4}
\begin{cases}
{\rm{PSU}}_2\rightarrow{\rm{PSU}}_2\times{\rm{PSU}}_2\\
\alpha\mapsto(\alpha,\alpha^{-1})
\end{cases}
\end{equation}
and 
\begin{equation}\label{auxiliary5}
\begin{cases}
{\rm{PSU}}_2\times{\rm{PSU}}_2\rightarrow\Bbb{CP}^1\\
(\alpha,\beta)\mapsto\alpha\circ f_0\circ\beta|_{z=\infty}.
\end{cases}
\end{equation}
Hence, the homomorphism that \eqref{auxiliary3} induces on $\pi_3$ can be written as a composition accordingly. Identifying $\pi_3\left({\rm{PSU}}_2\right)$ with
$$
\pi_3\left({\rm{SU}}_2=S^3\right)\cong\Bbb{Z},
$$
the map \eqref{auxiliary4} induces 
$$
\Bbb{Z}\rightarrow\Bbb{Z}\times\Bbb{Z}
$$
$$
n\mapsto (n,-n)
$$
on $\pi_3$ because the inverse map $\alpha\mapsto \alpha^{-1}$ on the odd-dimensional Lie group ${\rm{SU}}_2$ is of degree $-1$.  Switching to \eqref{auxiliary5}, restrictions to the first and the second components of the domain are homotopic to
$$
\begin{cases}
{\rm{PSU}}_2\rightarrow\Bbb{CP}^1\\
\alpha\mapsto\alpha(\infty)
\end{cases}
$$
and 
$$
\begin{cases}
{\rm{PSU}}_2\rightarrow\Bbb{CP}^1\\
\beta\mapsto f_0\left(\beta(\infty)\right)
\end{cases}
$$
respectively.
Pre-composing with the covering map ${\rm{SU}}_2\rightarrow{\rm{PSU}}_2$, we need to analyze the homomorphisms induced on $\pi_3$ by  
\begin{equation}\label{auxiliary6}
\begin{cases}
{\rm{SU}}_2=\left\{\begin{bmatrix}
r&-\bar{t}\\
t&\bar{r}
\end{bmatrix}\,\Bigg|\, |r|^2+|t|^2=1\right\}\rightarrow\Bbb{CP}^1\\
\begin{bmatrix}
r&-\bar{t}\\
t&\bar{r}
\end{bmatrix}\mapsto\frac{r}{t}
\end{cases}
\end{equation}
and
\begin{equation}\label{auxiliary7}
\begin{cases}
{\rm{SU}}_2=\left\{\begin{bmatrix}
r&-\bar{t}\\
t&\bar{r}
\end{bmatrix}\,\Bigg|\, |r|^2+|t|^2=1\right\}\rightarrow\Bbb{CP}^1\\
\begin{bmatrix}
r&-\bar{t}\\
t&\bar{r}
\end{bmatrix}\mapsto f_0\left(\frac{r}{t}\right).
\end{cases}
\end{equation}
The first one is the Hopf Fibration which gives rise to an isomorphism 
$$
\pi_3\left({\rm{SU}}_2=S^3\right)\rightarrow\pi_3\left(\Bbb{CP}^1=S^2\right)
$$
of infinite cyclic groups. As for \eqref{auxiliary7}, the map is the post-composition of the Hopf Fibration \eqref{auxiliary6} with a degree $d$ map $f_0:\Bbb{CP}^1\rightarrow\Bbb{CP}^1$ and thus amounts to multiplication by $d^2$ on $\pi_3$ (cf. \cite[chap. 4.B]{MR1867354}). Putting these all together, we see that the homomorphism induced by \eqref{auxiliary5} on $\pi_3$ can be described as 
$$
\Bbb{Z}\times\Bbb{Z}\rightarrow\Bbb{Z}
$$
$$
(m,n)\mapsto m+d^2n.
$$
We conclude that on the third homotopy groups \eqref{auxiliary3} induces 
\begin{equation}\label{auxiliary8}
\pi_3({\rm{PSL}}_2(\Bbb{C}))\cong\Bbb{Z}\rightarrow\pi_3(\Bbb{CP}^1)\cong\Bbb{Z}\end{equation}
$$
n\mapsto(1-d^2)n
$$
and thus, an isomorphism 
$\pi_3({\rm{PSL}}_2(\Bbb{C}))\otimes\Bbb{Q}\stackrel{\sim}{\rightarrow}\pi_3(\Bbb{CP}^1)\otimes\Bbb{Q}.$
\\
\indent
The final step is to establish isomorphism \eqref{main isomorphism} in the case where $d=2$ or $d=3$. The conjugation action \eqref{auxiliary1} of the group of M\"obius transformations is compatible with  iteration; hence, we have the commutative diagram below
\begin{equation*}
\xymatrixcolsep{5pc}\xymatrix{
{\rm{PSL}}_2(\Bbb{C})\ar[dr]_{\text{orbit inclusion}\quad}\ar[r]^{\quad\text{orbit inclusion}}&{\rm{Rat}}_d\ar[d]^{f\mapsto f\circ f}\\
&{\rm{Rat}}_{d^2}
}    
\end{equation*}
and the corresponding diagram of homology groups:
\begin{equation*}
\xymatrix{
H_3({\rm{PSL}}_2(\Bbb{C});\Bbb{Q})\ar[dr]\ar[r]& H_3({\rm{Rat}}_d;\Bbb{Q})\ar[d]\\
& H_3\left({\rm{Rat}}_{d^2};\Bbb{Q}\right).
}    
\end{equation*}
As $d^2>4$, we already know that the diagonal arrow $H_3({\rm{PSL}}_2(\Bbb{C});\Bbb{Q})\rightarrow H_3\left({\rm{Rat}}_{d^2};\Bbb{Q}\right)$ is an isomorphism. This implies the same for
$H_3({\rm{PSL}}_2(\Bbb{C});\Bbb{Q})\rightarrow H_3({\rm{Rat}}_d;\Bbb{Q})$ because all of the vector spaces involved are one-dimensional. 
\end{proof}

\begin{remark}
Lemma \ref{main lemma} can also be used to exhibit a cycle generating  $H_3({\rm{Rat}}_d;\Bbb{Q})$. It suffices to consider the orbit of a degree $d$ map, say $f(z)=z^d$, under the conjugation action of M\"obius transformation corresponding to matrices belonging to 
$$
{\rm{SU}}_2=\left\{\begin{bmatrix}
r&-\bar{t}\,\\
t&\bar{r}\,
\end{bmatrix}\,\bigg|\, |r|^2+|t|^2=1\right\}.
$$
Thinking of $S^3$ as the unit sphere in $\Bbb{C}^2$, we obtain the cycle 
$$
(r,t)\in S^3\mapsto
\frac{\sum_{i=0}^d\binom{d}{i}\left(r\bar{r}^i\bar{t}^{d-i}+(-1)^{i+1}\bar{t}t^ir^{d-i}\right)z^i}
{\sum_{i=0}^d\binom{d}{i}\left(t\bar{r}^i\bar{t}^{d-i}+(-1)^i\bar{r}t^ir^{d-i}\right)z^i}
\in{\rm{Rat}}_d
$$
which generates $H_3({\rm{Rat}}_d;\Bbb{Q})$.
\end{remark}
\begin{remark}
While the moduli space $\mathcal{M}_d={\rm{Rat}}_d\big/{\rm{SL}}_2(\Bbb{C})$ is built from ${\rm{Rat}}_d$, this is not the only interesting subvariety of the projective space $\Bbb{CP}^{2d+1}$ of pairs of polynomials of degree at most $d$. Indeed, there are two other natural subvarieties determined by the conjugation action of ${\rm{SL}}_2(\Bbb{C})$ (or ${\rm{PSL}}_2(\Bbb{C})$); the loci of stable and semistable points   $\left({\Bbb{CP}}^{2d+1}\right)^{\rm{s}}\subseteq\left({\Bbb{CP}}^{2d+1}\right)^{\rm{ss}}$ studied by Silverman \cite{MR1635900}. They lie strictly between ${\rm{Rat}}_d$ and $\Bbb{CP}^{2d+1}$ and give rise to the geometric quotient 
$\mathcal{M}_d^{\rm{s}}:=\left({\Bbb{CP}}^{2d+1}\right)^{\rm{s}}\big/{\rm{SL}}_2(\Bbb{C})$ as well as  the geometric invariant theory quotient $\mathcal{M}_d^{\rm{ss}}$ for the action ${\rm{SL}}_2(\Bbb{C})\curvearrowright\left({\Bbb{CP}}^{2d+1}\right)^{\rm{ss}}$. The latter is a projective variety containing $\mathcal{M}_d$ and $\mathcal{M}_d^{\rm{s}}$. It is interesting to note that our results do not carry over  to these other quotients.\\
\indent
Theorem \ref{main 1} on the rational acyclicity of $\mathcal{M}_d$ can fail for $\mathcal{M}_d^{\rm{s}}$ and $\mathcal{M}_d^{\rm{ss}}$. For instance, in degree two, $\mathcal{M}_2^{\rm{s}}=\mathcal{M}_2^{\rm{ss}}$ can be identified with the projective plane $\Bbb{CP}^2$ \cite[Theorem 6.1]{MR1635900}. The same is true about Theorem \ref{main 2}. Unlike 
the case of $\mathcal{M}_d$, the spaces $\mathcal{M}_d^{\rm{s}}$ and $\mathcal{M}_d^{\rm{ss}}$ are simply connected for any choice of $d$. For the compactification $\mathcal{M}_d^{\rm{ss}}$, this follows from \cite{MR3348559}. As for the geometric quotient $\mathcal{M}_d^{\rm{s}}$, the homomorphism $\pi_1\left(\left({\Bbb{CP}}^{2d+1}\right)^{\rm{s}}\right)\rightarrow\pi_1(\mathcal{M}_d^{\rm{s}})$ is surjective (cf. \cite{MR2852978}) and we claim that -- unlike ${\rm{Rat}}_d$ -- the variety $\left({\Bbb{CP}}^{2d+1}\right)^{\rm{s}}$ is simply connected. To see this, notice that the closed subvariety  
$\Bbb{CP}^{2d+1}-\left({\Bbb{CP}}^{2d+1}\right)^{\rm{s}}$ of the resultant hypersurface is a proper subvariety as it does not contain points from $\left({\Bbb{CP}}^{2d+1}\right)^{\rm{s}}-{\rm{Rat}}_d$, and is thus of real codimension at least four in ${\Bbb{CP}}^{2d+1}$ due to the irreducibility of the resultant hypersurface (compare with the argument in \cite[Lemma 6.4]{MR3709645}). Consequently, $\left({\Bbb{CP}}^{2d+1}\right)^{\rm{s}}$ and $\Bbb{CP}^{2d+1}$ have the same fundamental groups and the same homology groups up to dimension two. Notice  moreover that this means -- in contrast with \eqref{fibration5} -- that for the orbifold fiber bundle 
$$
{\rm{PSL}}_2(\Bbb{C})\rightarrow\left({\Bbb{CP}}^{2d+1}\right)^{\rm{s}}\rightarrow\mathcal{M}_d^{\rm{s}}    
$$
the rational cohomology of the total space differs from that of the fiber because  $H^2\left(\left({\Bbb{CP}}^{2d+1}\right)^{\rm{s}};\Bbb{Q}\right)$ is isomorphic to $H^2(\Bbb{CP}^{2d+1};\Bbb{Q})$ which is non-trivial.
\end{remark}

\begin{remark}
The generalized Leray-Hirsch theorem is not strictly required to prove Theorem \ref{main 1}. Indeed, the Borel construction\footnote{For a topological group $G$ acting on a topological space $X$, we will denote the homotopy quotient (not to be confused with the GIT quotient) by $X\hcoker G$. It is given by $X\times_G \mathrm{E}G$ where $\mathrm{E}G$ is a contractible space on which $G$ acts freely and fits into a fiber sequence $X\to X\hcoker G \to \mathrm{B}G$.} ${\rm{Rat}}_d\hcoker{\rm{PSL}}_2(\Bbb{C})$ -- which is fibered over  $\mathrm{BPSL}_2(\Bbb{C})$ -- has the same rational cohomology as $\mathcal{M}_d$. Therefore, it is enough to show that ${\rm{Rat}}_d\hcoker{\rm{PSL}}_2(\Bbb{C})$ is rationally acyclic which can be similarly deduced from Lemma \ref{main lemma} and the Serre spectral sequence for the fibration given by the Borel construction \[{\rm{Rat}}_d\hcoker{\rm{PSL}}_2(\Bbb{C})\to \mathrm{BPSL}_2(\Bbb{C}).\]
To see why the natural map 
\[
{\rm{Rat}}_d\hcoker{\rm{PSL}}_2(\Bbb{C})\to \mathcal{M}_d,
\]
is a rationally acyclic map, note that by Lemma \ref{proper} and the discussion in \S\ref{background2} the group $\mathrm{PSL}_2(\Bbb{C})$ acts properly on $\mathrm{Rat}_d$ with finite stabilizer subgroups. It then follows from the Palais slice theorem \cite{MR0126506} that every orbit of the conjugation action of $\mathrm{PSL}_2(\Bbb{C})$ on $\mathrm{Rat}_d$ is an equivariant neighborhood retract. Now, let $p$ be a prime number that does not divide the order of the stabilizer subgroups and let $\pi$ denote the natural map between the homotopy quotient and the quotient
\[
\pi:{\rm{Rat}}_d\hcoker{\rm{PSL}}_2(\Bbb{C})\to {\rm{Rat}}_d\slash{\rm{PSL}}_2(\Bbb{C})=\mathcal{M}_d.
\]
This is not a fiber bundle but the fiber $\pi^{-1}(\langle f\rangle)$ of  each orbit $\langle f\rangle\in \mathcal{M}_d$ is the classifying space $\mathrm{BStab}(f)$. As such, we consider  the sheaf $R^n\pi_*  \Bbb{Z}[\frac{1}{p}]$ which is the sheaf associated with the presheaf 
\[
U\mapsto H^n\left(\pi^{-1}(U);\,\Bbb{Z}[\frac{1}{p}]\right).
\]
Its stalk at the orbit $\langle f\rangle$ is given by the colimit over the open neighborhoods $U$ containing $\langle f\rangle$. Given that every orbit is an equivariant neighborhood retract, $\pi^{-1}(U)$ deformation retracts to $\pi^{-1}(\langle f\rangle)$ for a cofinal of open neighborhoods $U$ of $\langle f\rangle$. Therefore, the stalk of $R^n\pi_*  \Bbb{Z}[\frac{1}{p}]$ at $\langle f\rangle$ is the same as $H^n(\mathrm{BStab}(f);\Bbb{Z}[\frac{1}{p}])$ which is zero for $n>0$ because $p$ does not divide the order of  $\mathrm{Stab}(f)$. Hence, the higher direct images $R^n\pi_*  \Bbb{Z}[\frac{1}{p}]$ are all trivial for $n>0$. Thus, the Leray spectral sequence for the map $\pi$ in $\Bbb{Z}[\frac{1}{p}]$-cohomology collapses which implies that ${\rm{Rat}}_d\hcoker{\rm{PSL}}_2(\Bbb{C})$ and $\mathcal{M}_d$ have the same $\Bbb{Z}[\frac{1}{p}]$-cohomology and, in particular, the same rational cohomology. The former space is rationally acyclic; and Theorem \ref{main 1} then follows.\\
\indent
On the other hand, it is easy to see that $\pi$ does not induce a weak homotopy equivalence. There are orbits with cyclic stabilizer groups. Let $\langle f\rangle$ be one of those orbits. The composition of the following maps
\[
\mathrm{BStab}(f)\xrightarrow{\iota_f} {\rm{Rat}}_d\hcoker{\rm{PSL}}_2(\Bbb{C})\xrightarrow{\pi}\mathcal{M}_d,
\]
is null-homotopic. We observe however that $\iota_f$ induces a non-trivial map on cohomology. Therefore, $\pi$ cannot be a weak homotopy equivalence. Consider the composition 
\[
\mathrm{BStab}(f)\xrightarrow{\iota_f} {\rm{Rat}}_d\hcoker{\rm{PSL}}_2(\Bbb{C})\to \mathrm{BPSL}_2(\Bbb{C}).
\]
Since $\mathrm{Stab}(f)$ sits in the maximal torus of $\mathrm{PSL}_2(\Bbb{C})$,  the induced map 
\[
H^*(\mathrm{BPSL}_2(\Bbb{C});\Bbb{Z})\to H^*(\mathrm{BStab}(f);\Bbb{Z})
\]
is nontrivial \cite[Theorem 2]{MR124050}. Hence, $\pi$ does  not induce a homotopy equivalence between ${\rm{Rat}}_d\hcoker{\rm{PSL}}_2(\Bbb{C})$ and $\mathcal{M}_d$. 
In fact, these spaces are not homotopy equivalent since their fundamental groups are different by Theorem \ref{main 2}.  
\end{remark}

\subsection{Application to marked moduli spaces of rational maps}\label{application marked}
Every non-trivial rational cohomology group of ${\rm{PSL}}_2(\Bbb{C})$ is one-dimensional and, hence, our strategy from \S\ref{rational homology} can be applied to several other orbifolds which admit maps to $\mathcal{M}_d$. 
\begin{proposition}\label{proposition marked}
Let ${\rm{N}}$ be a topological space equipped with a continuous action of the group 
$G={\rm{PSL}}_2(\Bbb{C}) \text{ or } {\rm{SL}}_2(\Bbb{C})$ 
with the property that the stabilizers of the action are finite 
and 
$$
{\rm{N}}\rightarrow\mathcal{N}:={\rm{N}}/G
$$
is an orbifold fiber bundle in the sense of \cite{MR2078574}; e.g. ${\rm{N}}$ is a smooth variety and ${\rm{N}}\rightarrow\mathcal{N}$ 
is a geometric quotient for the action of $G$ in the sense of GIT. Moreover, suppose that there is a map ${\rm{N}}\rightarrow{\rm{Rat}}_d$
that respects the $G$-actions where $G$ acts on the ${\rm{PSL}}_2(\Bbb{C})$-space ${\rm{Rat}}_d$ via the obvious homomorphism $G\rightarrow{\rm{PSL}}_2(\Bbb{C})$. In such a situation one has the following isomorphism of graded $\Bbb{Q}$-vector spaces:
$$
H^*({\rm{N}};\Bbb{Q})\cong H^*(\mathcal{N};\Bbb{Q})\otimes H^*({\rm{PSL}}_2(\Bbb{C});\Bbb{Q}).
$$
\end{proposition}

\begin{proof}
The canonical projection ${\rm{SL}}_2(\Bbb{C})\rightarrow{\rm{PSL}}_2(\Bbb{C})$ induces isomorphisms of rational cohomology groups in all dimensions so we may safely assume that
$G={\rm{PSL}}_2(\Bbb{C})$. By the same logic used in the proof of Theorem \ref{main 1}, the desired isomorphism can be obtained by applying  the generalized Leray-Hirsch theorem  
\cite[Theorem 2]{MR2078574} to 
$$
{\rm{PSL}}_2(\Bbb{C})\rightarrow{\rm{N}}\rightarrow\mathcal{N}:={\rm{N}}\big/{\rm{PSL}}_2(\Bbb{C}) 
$$
provided that the fiber inclusion induces a surjection $H^3({\rm{N}};\Bbb{Q})\twoheadrightarrow H^3({\rm{PSL}}_2(\Bbb{C});\Bbb{Q})$. This follows from the commutative diagram below where the diagonal arrow is an isomorphism by virtue of Lemma \ref{main lemma}
\begin{equation*}
\xymatrix{
H^3({\rm{Rat}}_d;\Bbb{Q})\ar[d]\ar[dr]^{\sim}& \\
H^3({\rm{N}};\Bbb{Q})\ar[r]& H^3\left({\rm{PSL}}_2(\Bbb{C});\Bbb{Q}\right)\cong\Bbb{Q}.
}   
\end{equation*}
\end{proof}
An immediate example of ${\rm{PSL}}_2(\Bbb{C})$-spaces admitting equivariant  maps to ${\rm{Rat}}_d$ is provided by ``marked'' moduli spaces of rational maps \cite[\S 6]{MR1246482}. One can work with  the \textit{fixed-point marked}, \textit{critically marked}  or \textit{totally marked} parameter spaces of rational maps of degree $d$ in which respectively the fixed points, the critical points or both fixed points and critical points are specified as ordered tuples  (of sizes $d+1$, $2d-2$ or $3d-1$) of points on the Riemann sphere. These parameter spaces and the corresponding moduli spaces (i.e. their quotients under the diagonal action of ${\rm{PSL}}_2(\Bbb{C})$) are denoted by  ${\rm{Rat}}^{\rm{fm}}_d$, ${\rm{Rat}}^{\rm{cm}}_d$, ${\rm{Rat}}^{\rm{tm}}_d$ 
and $\mathcal{M}^{\rm{fm}}_d$, $\mathcal{M}^{\rm{cm}}_d$, $\mathcal{M}^{\rm{tm}}_d$ respectively. 

\begin{corollary}\label{corollary marked}
We have the corresponding isomorphisms of graded $\Bbb{Q}$-vector spaces:
\begin{equation*}
\begin{split}
&H^*({\rm{Rat}}^{\rm{fm}}_d;\Bbb{Q})\cong H^*(\mathcal{M}^{\rm{fm}}_d;\Bbb{Q})\otimes H^*({\rm{PSL}}_2(\Bbb{C});\Bbb{Q});\\
&H^*({\rm{Rat}}^{\rm{cm}}_d;\Bbb{Q})\cong H^*(\mathcal{M}^{\rm{cm}}_d;\Bbb{Q})\otimes H^*({\rm{PSL}}_2(\Bbb{C});\Bbb{Q});\\
&H^*({\rm{Rat}}^{\rm{tm}}_d;\Bbb{Q})\cong H^*(\mathcal{M}^{\rm{tm}}_d;\Bbb{Q})\otimes H^*({\rm{PSL}}_2(\Bbb{C});\Bbb{Q}).\\
\end{split}    
\end{equation*}
\end{corollary}

\begin{proof}
This follows at once by Proposition \ref{proposition marked}.
\end{proof}

\begin{example}\label{critically marked example}
The critically-marked parameter space 
$${\rm{Rat}}^{\rm{cm}}_2=\left\{(f;c_1,c_2)\mid f\text{ a quadratic rational map with critical points }c_1,c_2\right\}$$
is an unramified double cover of ${\rm{Rat}}_2$. 
The corresponding marked moduli space  
$\mathcal{M}^{\rm{cm}}_2$
is of the homotopy type of $S^2$ \cite[Remark 6.5]{MR1246482}. With the help of Corollary \ref{corollary marked}, we obtain the Betti numbers of 
${\rm{Rat}}^{\rm{cm}}_2$
as 
$$
b_i({\rm{Rat}}^{\rm{cm}}_2)=
\begin{cases}
1\quad i=0,2,3,5,\textit{ and}\\
0\quad \textit{otherwise}.
\end{cases}
$$
\end{example}

\begin{example}\label{fixed point marked example}
In his paper on quadratic rational maps, Milnor asks about the homology of the fixed-point marked spaces ${\rm{Rat}}^{\rm{fm}}_2$ and $\mathcal{M}^{\rm{fm}}_2$ 
\cite[Remark 6.7]{MR1246482}. He identifies the latter space with the affine surface
$$
\left\{(\mu_1,\mu_2,\mu_3)\in\Bbb{C}^3\,\big|\,\mu_1+\mu_2+\mu_3=\mu_1\mu_2\mu_3+2\right\}.
$$
One can solve for $\mu_3$ as $$\mu_3=\frac{2-(\mu_1+\mu_2)}{1-\mu_1\mu_2}$$ provided that $\mu_1\mu_2\neq 1$. If $\mu_1\mu_2=1$, we must have $\mu_1=\mu_2=1$ and $\mu_3$ can then be arbitrary. Consequently, $\mathcal{M}^{\rm{fm}}_2$ admits a closed subvariety homeomorphic to $\Bbb{C}$, and collapsing this contractible subvariety results in the homotopy equivalent space $\Bbb{C}^2\big/\{\mu_1\mu_2=1\}$. The curve
$\mu_1\mu_2=1$ collapsed here is of the homotopy type of circle and thus:
$$
\tilde{H}_i(\mathcal{M}^{\rm{fm}}_2;\Bbb{Z})\cong\tilde{H}_i\left(\Bbb{C}^2\big/\{\mu_1\mu_2=1\};\Bbb{Z}\right)\cong\tilde{H}_{i-1}(S^1;\Bbb{Z}). 
$$
Therefore, 
$$
b_i(\mathcal{M}^{\rm{fm}}_2)=
\begin{cases}
1\quad i=0,2,\textit{ and}\\
0\quad \textit{otherwise},
\end{cases}
$$
and by invoking Corollary \ref{corollary marked} we obtain
$$
b_i({\rm{Rat}}^{\rm{fm}}_2)=
\begin{cases}
1\quad i=0,2,3,5\textit{, and}\\
0\quad \textit{otherwise}.
\end{cases}
$$
\end{example}

\begin{example}\label{totally marked example}
The totally marked parameter space ${\rm{Rat}}^{\rm{tm}}_2$ is the space of all sextuples 
$$(f;z_1,z_2,z_3;c_1,c_2)$$
where $f$ is a quadratic rational map with critical points $c_1,c_2$ and fixed points $z_1,z_2,z_3$ (each repeated according to its multiplicity). The diagonal action of the group of M\"obius transformations on ${\rm{Rat}}^{\rm{tm}}_2$ is free and its orbit space, the totally marked moduli space ${\mathcal{M}}^{\rm{tm}}_2$, is a smooth complex surface; cf. \cite{MR3316459}. 
The projection map ${\mathcal{M}}^{\rm{tm}}_2\rightarrow{\mathcal{M}}^{\rm{fm}}_2$ is two-to-one  and ramifies only above the unique singular point of  ${\mathcal{M}}^{\rm{fm}}_2$ \cite[p. 51]{MR1246482}. 
We claim that the Betti numbers of ${\mathcal{M}}^{\rm{tm}}_2$ (and thus those of ${\rm{Rat}}^{\rm{tm}}_2$) are the same as the Betti numbers of ${\mathcal{M}}^{\rm{fm}}_2$
(respectively, of ${\rm{Rat}}^{\rm{fm}}_2$) which were calculated in the previous example. If the first fixed point $z_1$ of $f$ is one of the critical points $c_1$ or $c_2$, the quadratic rational map $f$ possesses a \textit{super-attracting} fixed point and is conjugate to a unique quadratic polynomial of the form $z^2+b$. Thus, the corresponding class in ${\mathcal{M}}^{\rm{tm}}_2$ belongs to either 
$$
\left\{\langle(z^2+(r-r^2);\infty,r,1-r;\infty,0)\rangle\,\big|\, r\in\Bbb{C}\right\}
$$
or 
$$
\left\{\langle(z^2+(r-r^2);\infty,r,1-r;0,\infty)\rangle\,\big|\, r\in\Bbb{C}\right\}.
$$
These are two disjoint closed subvarieties of ${\mathcal{M}}^{\rm{tm}}_2$ each homeomorphic to $\Bbb{C}$. On the other hand, if the fixed point $z_1$ is not a critical point, by putting it at $\infty$ and taking the critical points to be $\pm1$, one can write $f$ in the \textit{mixed normal form} 
$f(z)=\frac{1}{\mu}\left(z+\frac{1}{z}\right)+a$ \cite[Appendix C]{MR1246482}. Here, $\mu$ is the expansion factor by which $f$ acts on the tangent plane at $\infty$. The other two fixed points are different from $0$. Writing them as $z_2=\frac{1}{s}$, $z_3=\frac{1}{t}$ where $s,t\in\Bbb{C}$ and solving for $\mu$ and $a$, we arrive at the following open subset of ${\mathcal{M}}^{\rm{tm}}_2$:
$$
\left\{\left\langle\left(\frac{1}{1-st}\left(z+\frac{1}{z}\right)-\frac{s+t}{1-st};\infty,\frac{1}{s},\frac{1}{t};+1,-1\right)\right\rangle\,\bigg|\, s,t\in\Bbb{C}\, \text{ and }st\neq 1\right\}.\footnote{It is worth mentioning that these computations yield an explicit description of ${\mathcal{M}}^{\rm{tm}}_2$ in $\Bbb{C}^3$ as the affine surface $(1-st)u+(s^2-1)=0$. Finding the \textit{multipliers} of fixed points, one obtains
$\mu_1=\mu=1-st$, $\mu_2=f'\left(\frac{1}{s}\right)=\frac{1-s^2}{1-st}=u$ and $\mu_3=f'\left(\frac{1}{t}\right)=\frac{1-t^2}{1-st}=1+st-t^2\mu_2$. Thus, in these models for ${\mathcal{M}}^{\rm{tm}}_2$ and ${\mathcal{M}}^{\rm{fm}}_2$, the morphism  
${\mathcal{M}}^{\rm{tm}}_2\rightarrow{\mathcal{M}}^{\rm{fm}}_2$
can be described as:
$(s,t,u)\mapsto(\mu_1,\mu_2,\mu_3)=(1-st,u,1+st-t^2u).$}
$$
This is a copy of $\Bbb{C}^2-\{st=1\}$ and as $\mu=1-st\to 0$, we tend to the previous cases where the fixed point $\infty$ was critical too. Consequently, after collapsing the aforementioned disjoint copies of $\Bbb{C}$ in ${\mathcal{M}}^{\rm{tm}}_2$, we obtain the space $\Bbb{C}^2\big/\{st=1\}$ which is of the same homotopy type. We observed in Example \ref{fixed point marked example} that this is the homotopy type of ${\mathcal{M}}^{\rm{fm}}_2$ as well. 
\end{example}

\subsection{Variants \texorpdfstring{$\mathcal{M}_d^{\rm{pre}}$}{Mdpre}  and \texorpdfstring{$\mathcal{M}_d^{\rm{post}}$}{Mdpost} of the moduli space \texorpdfstring{$\mathcal{M}_d$}{Md}}\label{application variant}
Aside from the conjugation action we have studied so far, ${\rm{PSL}}_2(\Bbb{C})$  also acts from the left on ${\rm{Rat}}_d$ through pre- and post-composition with M\"obius transformations. The corresponding quotients will be denoted by
\begin{equation}\label{variants}
\mathcal{M}_d^{\rm{pre}}:={\rm{Rat}}_d\big/f\sim f\circ\alpha^{-1},\quad\text{and}\quad
\mathcal{M}_d^{\rm{post}}:={\rm{Rat}}_d\big/f\sim \alpha\circ f
\text{ where } \alpha\in{\rm{PSL}}_2(\Bbb{C}).
\end{equation}
The action by post-composition is free and thus $\mathcal{M}_d^{\rm{post}}$ fits into a genuine fiber bundle
\begin{equation}\label{fibration1'}
{\rm{PSL}}_2(\Bbb{C})\hookrightarrow{\rm{Rat}}_d\rightarrow\mathcal{M}_d^{\rm{post}}. 
\end{equation}
On the other hand, $\mathcal{M}_d^{\rm{pre}}$ is the moduli space of degree $d$ branched covers $\Bbb{CP}^1\rightarrow\Bbb{CP}^1$ and the action by pre-composition is not free since there may exist non-trivial deck transformations. In this case, we have an orbifold fiber bundle 
\begin{equation}\label{fibration1''}
{\rm{PSL}}_2(\Bbb{C})\rightarrow{\rm{Rat}}_d\rightarrow\mathcal{M}_d^{\rm{pre}}. 
\end{equation}
Once again, there are also pre- and post-composition actions of ${\rm{SL}}_2(\Bbb{C})$ on the universal cover $\widetilde{\rm{Rat}}_d$ that, analogously to \eqref{action2}, are given by
\begin{equation}\label{action2 pre-variant}
\begin{bmatrix}
r&s\\
t&u
\end{bmatrix}\cdot
\left(P(X,Y),Q(X,Y)\right):=\left(P(uX-sY,-tX+rY),Q(uX-sY,-tX+rY)\right),\, and   
\end{equation}
\begin{equation}\label{action2 post-variant}
\begin{bmatrix}
r&s\\
t&u
\end{bmatrix}\cdot\left(P(X,Y),Q(X,Y)\right):=
\left(rP(X,Y)+sQ(X,Y),tP(X,Y)+uQ(X,Y)\right).    
\end{equation}

The computations of  Lemma \ref{main lemma} now carry over and allow us to describe the rational homology of $\mathcal{M}_d^{\rm{pre}}$ and $\mathcal{M}_d^{\rm{post}}$.

\begin{lemma}\label{main lemma variant}
Assuming $d>3$, the inclusion of an orbit of the pre-composition or post-composition action induces isomorphisms in the rows of the commutative diagram \eqref{diagram}.
\end{lemma}

\begin{proof}
The proof of Lemma \ref{main lemma} carries over mutatis mutandis. The only difference is the calculation \eqref{auxiliary8} of the map $\pi_3({\rm{PSL}}_2(\Bbb{C}))\rightarrow\pi_3({\rm{Rat}}_d)$ which in these cases would be multiplication by $-d^2$ or an isomorphism respectively.
\end{proof}

\begin{theorem}\label{main variant}
For any $d>1$ the spaces $\mathcal{M}_d^{\rm{pre}}$ and $\mathcal{M}_d^{\rm{post}}$ are rationally acyclic. 
\end{theorem}

\begin{proof}
The idea once again is to apply the Leray-Hirsch theorem (for fiber bundles or orbifold fiber bundles) to \eqref{fibration1'} or \eqref{fibration1''}. It suffices to verify Lemma \ref{main lemma} for the inclusion of orbits $\left\{\alpha\circ f_0\mid\alpha\in {\rm{PSL}}_2(\Bbb{C})\right\}$ or $\left\{f_0\circ\alpha\mid\alpha\in {\rm{PSL}}_2(\Bbb{C})\right\}$ in ${\rm{Rat}}_d$. 
Lemma \ref{main lemma variant} establishes this for $d>3$. Analogously to the end of the proof of Lemma \ref{main lemma}, we use certain equivariant maps to construct isomorphisms for $d=2$ and $d=3$ as well. 
One could use maps of the from 
$${\rm{Rat}}_d\xrightarrow{{f\mapsto g\circ f}}{\rm{Rat}}_{kd}\quad\text{ or }\quad{\rm{Rat}}_d\xrightarrow{{f\mapsto f\circ g}}{\rm{Rat}}_{kd}$$  where $g$ is a fixed rational map of degree $k\geq 2$. These are equivariant for the pre-composition and post-composition actions respectively.
\end{proof}
\begin{remark}\label{claudio}
Lemma \ref{main lemma variant} was independently established in greater generality by Claudio G\'omez-Gonz\'ales for post-composition. 
See \cite{MR4140094} for  the corresponding statement for the space of holomorphic maps of degree $d$ from $\mathbb{CP}^n$ to $\mathbb{CP}^m$ and some generalizations of Theorem \ref{main variant} in higher dimensions.
\end{remark}

\begin{example}
The moduli space $\mathcal{M}_2^{\rm{pre}}$ can be  described  explicitly. The two critical points of an arbitrary quadratic rational map $f$ may be assumed to be $0$ and $\infty$ after pre-composing with an appropriate M\"obius transformation. The map can then be written as 
$f(z)=\frac{rz^2+s}{tz^2+u}$
whose coefficients come from a unimodular matrix 
$$\begin{bmatrix}
r&s\\
t&u
\end{bmatrix}.$$
As such, $\mathcal{M}_2^{\rm{pre}}$ may be thought of as the space of all such rational maps modulo the subgroup 
$$
\left\{z\mapsto\lambda z^{\pm 1}\mid\lambda\in\Bbb{C}^*\right\}\cong\left\{z\mapsto\lambda z\mid\lambda\in\Bbb{C}^* \right\}\rtimes\left\langle z\mapsto\frac{1}{z}\right\rangle 
$$
of M\"obius transformations that preserve $\{0,\infty\}$. Pre-composing with $z\mapsto\lambda z$ or $z\mapsto\frac{1}{z}$ amounts to the following transformations of unimodular matrices respectively:
$$
\begin{bmatrix}
r&s\\
t&u
\end{bmatrix}\mapsto
\begin{bmatrix}
\lambda r&\frac{s}{\lambda}\\
\lambda t&\frac{u}{\lambda}
\end{bmatrix},\quad
\begin{bmatrix}
r&s\\
t&u
\end{bmatrix}\mapsto
\begin{bmatrix}
{\rm{i}}s&{\rm{i}}r\\
{\rm{i}}u&{\rm{i}}t
\end{bmatrix}.
$$
The invariants of the first action are given by $a:=rs$, $b:=tu$ and $c:=ru=st+1$ which are related by  $ab=c(c-1)$. The second transformation above changes them via 
$a\mapsto -a$, $b\mapsto -b$ and $c\mapsto 1-c$. Hence, $\mathcal{M}_2^{\rm{pre}}$ can be identified with the quotient of the affine surface 
$$
V:=\left\{(a,b,c)\in\Bbb{C}^3\,\big|\, ab=c(c-1)\right\}
$$
by the involution
$(a,b,c)\mapsto (-a,-b,1-c)$. After the affine change of coordinates 
$$
(a',b',c'):=\left({\rm{i}}(a+b),a-b,2c-1\right),
$$
this translates to the involution $(a',b',c')\mapsto(-a',-b',-c')$
acting on 
$$
V':=\left\{(a',b',c')\in\Bbb{C}^3\,\big|\, a^{\prime\,2}+b^{\prime\,2}+c^{\prime\,2}=1\right\}.
$$
It is not hard to see that $V'$ is homeomorphic to the tangent bundle of the unit sphere in $\Bbb{R}^3$. Its quotient by this involution is homeomorphic to the tangent bundle of $\Bbb{RP}^2$. We deduce that $\mathcal{M}_2^{\rm{pre}}$ is of the homotopy type of $\Bbb{RP}^2$ and that it is rationally acyclic while $H^*(\mathcal{M}_2^{\rm{pre}};\Bbb{Z}/2\Bbb{Z})$ is non-trivial and 
$\pi_1\left(\mathcal{M}_2^{\rm{pre}}\right)\cong\Bbb{Z}/2\Bbb{Z}$ (cf. Proposition \ref{variant fundamental group}).
\end{example}

It should be mentioned at this point that the rational acyclicity of $\mathcal{M}_d^{\rm{post}}$ also appears in \cite[Theorem C]{MR1445557}. There, this space is shown to be homeomorphic to the \textit{space of projective equivalence classes of non-singular Toeplitz matrices of dimension $d$}. We record this fact for the later use. 

\begin{definition-theorem}[{\cite[Theorem A]{MR1445557}}]\label{Toeplitz}
Let $\mathcal{T}_d$ be the space of non-singular $d\times d$ complex Toeplitz matrices 
\begin{equation}\label{Toeplitz matrix}
\begin{bmatrix}
a_d&a_{d+1}&\dots&a_{2d-2}&a_{2d-1}\\
a_{d-1}&a_d&\dots&a_{2d-3}&a_{2d-2}\\
\vdots&\vdots&\ddots&\vdots&\vdots\\
a_2&a_3&\dots&a_d&a_{d+1}\\
a_1&a_2&\dots&a_{d-1}&a_d
\end{bmatrix}
\end{equation}
modulo the scaling action of $\Bbb{C}^*$. It is homeomorphic to the space $\mathcal{M}^{\rm{post}}_d$ of rational maps of degree $d$ modulo post-composition (cf. \eqref{variants}). 
\end{definition-theorem}
\begin{remark}
Milgram's construction of a homeomorphism $\mathcal{M}^{\rm{post}}_d\rightarrow\mathcal{T}_d$ works as follows: Consider a pair $\left(p(z),q(z)\right)$ of coprime polynomials with $\max\{\deg p, \deg q\}=d$. Form the quotient $\Bbb{C}[z]\big/z^{2d-1}\Bbb{C}[z]$ which is a vector space of dimension $2d-1$;  $\{\overline{1},\overline{z}\dots,\overline{z^{2d-2}}\}$ is an ordered basis for it. 
The elements of this space determined by polynomials 
$$
p(z),zp(z),\dots,z^{d-2}p(z)\quad \text{and}\quad q(z),zq(z),\dots,z^{d-2}q(z)
$$
are linearly independent since $p(z)$ and $q(z)$ are coprime. Identifying $\left(\Bbb{C}[z]\big/z^{2d-1}\Bbb{C}[z]\right)^*$ with $\Bbb{C}^{2d-1}$ via the ordered basis mentioned above, the dual of the subspace spanned by these $2d-2$ elements is one-dimensional. This yields a complex vector $(a_1,\dots,a_{2d-1})$ unique up to multiplication by a non-zero scalar. Notice that the subspace of $\Bbb{C}[z]\big/z^{2d-1}\Bbb{C}[z]$ constructed here does not change by
replacing $\left(p(z),q(z)\right)$ with $\left(rp(z)+sq(z),tp(z)+uq(z)\right)$ where 
$\det
\begin{bmatrix}
r&s\\
t&u
\end{bmatrix}\neq 0$. We thus have assigned a vector $(a_1,\dots,a_{2d-1})$, considered modulo the action of $\Bbb{C}^*$, to the class of the rational map $\frac{p(z)}{q(z)}$ in $\mathcal{M}^{\rm{post}}_d$. It turns out that the Toeplitz matrix \eqref{Toeplitz matrix} associated with this vector is non-singular, hence a map from  $\mathcal{M}^{\rm{post}}_d$ to the space $\mathcal{T}_d$ of projective classes of non-singular Toeplitz matrices; a map that can be proved to be a homeomorphism. See {\cite[\S1]{MR1445557}} for the details.
\end{remark}

Milgram \cite{MR1445557} also extensively analyzes the mod $p$ cohomology of $\mathcal{M}_d^{\rm{post}}\cong\mathcal{T}_d$. Similarly, we will relate the mod $p$ cohomology of the dynamical moduli space $\mathcal{M}_d$ to that of the space $\mathcal{T}_d$ in the next subsection.

\subsection{The cohomology of \texorpdfstring{$\mathcal{M}_d$}{Md} with finite coefficients}\label{finite homology}
The main goal of this subsection is to relate the homology/cohomology groups of $\mathcal{M}_d$ with coefficients in a finite field to those of the space $\mathcal{T}_d$ of Toeplitz matrices described above. Our reason for doing so is  to eventually show that $\mathcal{M}_d$ is not contractible for large enough $d$. Throughout this subsection, $p$ is a prime integer and $\Bbb{F}_p$ is the field with $p$ elements. We begin with the following:

\begin{theorem}\label{main finite}
If $d\geq 2$ is an integer and $p$ is a prime number not dividing $d^3-d$, then 
$H^*(\mathcal{M}_d;\Bbb{F}_p)$
and 
$H^*(\mathcal{T}_d;\Bbb{F}_p)$
are isomorphic as graded $\Bbb{F}_p$-algebras.
\end{theorem}
To motivate this theorem, it must be mentioned that $H^*(\mathcal{T}_d;\Bbb{F}_p)$ has a rich structure. In fact, each orbit in the post-composition action of ${\rm{PSL}}_2(\Bbb{C})$ on 
${\rm{Rat}}_d$ contains  based maps satisfying $\infty\mapsto 1$. The quotient space $\mathcal{M}_d^{\rm{post}}$ may thus be identified with the based parameter space
${\rm{Rat}}^*_d$ modulo the free action of the subgroup of M\"obius transformations fixing $1\in\Bbb{CP}^1$. The aforementioned subgroup is of the homotopy type of $S^1$. Therefore, 
$\mathcal{M}_d^{\rm{post}}$ is homotopy equivalent to ${\rm{Rat}}^*_d/S^1$. The post-composition action is free and thus ${\rm{Rat}}^*_d/S^1$ is homotopy equivalent to the corresponding homotopy quotient. Thus, by the Borel construction, it fits into a fibration 
\begin{equation}\label{fibration6}
{\rm{Rat}}^*_d\rightarrow{\rm{Rat}}^*_d/S^1\rightarrow {\rm{B}}S^1=\Bbb{CP}^\infty.
\end{equation}
This amounts to a spectral sequence 
$$
E^{i,j}_2=H^i\left(\Bbb{CP}^\infty;H^j\left({\rm{Rat}}^*_d;\Bbb{F}_p\right)\right)
\cong H^i(\Bbb{CP}^\infty;\Bbb{Z})\otimes H^j({\rm{Rat}}^*_d;\Bbb{F}_p)\Rightarrow H^{i+j}(\mathcal{T}_d;\Bbb{F}_p)
$$
which is analyzed in \cite{MR1445557}. On the other hand, the mod $p$ cohomology groups of ${\rm{Rat}}^*_d$ are studied in \cite{MRTotaro,MR1097023,MR1318152}.

Our proof of Theorem \ref{main finite} mimics the proofs presented for Lemma \ref{main lemma} and Theorem \ref{main 1}  in \S\ref{rational homology}. The rational homology or cohomology and also the homotopy groups therein must be replaced by their mod $p$ analogues and a mod $p$ version of the generalized Leray-Hirsch theorem must be applied to the orbifold fiber bundle \eqref{fibration5}. These steps require us to exclude some prime characteristics: 
\begin{itemize}
\item The argument in \cite[\S 2]{MR2078574} for a Leray-Hirsch theorem applicable to the rational homology of orbifold fiber bundles carries over to the mod $p$ setting provided that $p$ does not divide the order of stabilizers. Any stabilizer of the conjugation action ${\rm{PSL}}_2(\Bbb{C})\curvearrowright{\rm{Rat}}_d$ is a finite M\"obius group and thus is either cyclic, dihedral or one of the symmetry groups of platonic solids of orders $12$, $24$ and $60$. The dimension of the locus in $\mathcal{M}_d$ corresponding to each stabilizer type is known: There exists a cyclic stabilizer of order $m$ or a dihedral stabilizer of order $2m$ only if $m$ divides $d-1$, $d$ or $d+1$ \cite[Proposition 2.7 and Proposition 2.12]{MR3709645}; and a stabilizer of order $60$ (the group $A_5$) occurs only if $d\equiv\pm 1\pmod{5}$ \cite[Theorem 3.16]{MR3709645}. The integer $(d-1)d(d+1)=d^3-d$ is always divisible by $2$ and $3$ (the prime divisors of $A_4$ and $S_4$) and moreover by $5$ when $d\equiv\pm 1\pmod{5}$. Consequently, the integer $d^3-d$ is divisible by  prime divisors of all stabilizers involved.  
\item In analogy with $H^*({\rm{PSL}}_2(\Bbb{C});\Bbb{Q})$, we want the graded vector space $H^*({\rm{PSL}}_2(\Bbb{C});\Bbb{F}_p)$ to be of dimension one in degrees zero and three and zero-dimensional otherwise. Thus, $p=2$ must be thrown out. 
\item Just like the proof of Lemma \ref{main lemma}, we shall apply a version of the Hurewicz theorem for which we need $\pi_2(\widetilde{\rm{Rat}}_d)\otimes\Bbb{F}_p$ to be trivial and
$\pi_3(\widetilde{\rm{Rat}}_d)\otimes\Bbb{F}_p$ to be one-dimensional. Given \eqref{auxiliary2}, $p=2$ is the only problematic characteristic here.  
\item As in diagram \eqref{diagram}, we want the universal covering $\widetilde{\rm{Rat}}_d\rightarrow{\rm{Rat}}_d$ to induce an epimorphism 
$$
H_3(\widetilde{\rm{Rat}}_d;\Bbb{F}_p)\twoheadrightarrow H_3({\rm{Rat}}_d;\Bbb{F}_p).
$$
The degree of the covering is $2d$ (Proposition \ref{Rat fundamental}) and we should have $p\nmid 2d$ so that the transfer homomorphism provides us with a section. Hence, we want $p$ not to divide $2d$. 
\end{itemize}
All in all, the prime divisors of $d^3-d$ are the characteristics that we need to disregard in our proof.

\begin{proof}[Proof of Theorem \ref{main finite}]
We first claim that 
\begin{equation}\label{auxiliary14}
\begin{split}
&H^*({\rm{Rat}}_d;\Bbb{F}_p)\cong H^*(\mathcal{M}_d;\Bbb{F}_p)\otimes H^*({\rm{PSL}}_2(\Bbb{C});\Bbb{F}_p),\textit{ and}\\
&H^*({\rm{Rat}}_d;\Bbb{F}_p)\cong H^*(\mathcal{M}^{\rm{post}}_d;\Bbb{F}_p)\otimes H^*({\rm{PSL}}_2(\Bbb{C});\Bbb{F}_p)     
\end{split}    
\end{equation}
for any $d\geq 2$. This will be achieved by applying the Leray-Hirsch theorem (generalized or usual) to \eqref{fibration5} and \eqref{fibration1'}. It suffices to show that 
the linear maps 
\begin{equation}\label{auxiliary15}
\Bbb{F}_p\cong H_3({\rm{PSL}}_2(\Bbb{C});\Bbb{F}_p)\rightarrow H_3({\rm{Rat}}_d;\Bbb{F}_p)
\end{equation}
induced by orbit inclusions are injective. As in the proof of Lemma \ref{main lemma}, it is no loss of generality to take $d$ to be larger than three. The description \eqref{auxiliary2} of the first few homotopy groups of $\widetilde{{\rm{Rat}}}_d$ then implies that 
$$
\pi_2(\widetilde{\rm{Rat}}_d)\otimes\Bbb{F}_p=0 \quad\text{and}\quad \pi_3(\widetilde{\rm{Rat}}_d)\otimes\Bbb{F}_p\cong\Bbb{F}_p.
$$
We now invoke a mod $p$ version of the Hurewicz theorem as presented in \cite[Theorem 9.7]{MR2739026}. The space $\widetilde{\rm{Rat}}_d$ is simply connected with 
$\pi_2(\widetilde{\rm{Rat}}_d)\otimes\Bbb{F}_p=0$. As such,  
$H_2(\widetilde{\rm{Rat}}_d;\Bbb{F}_p)$ is trivial and, more importantly, the Hurewicz homomorphism
$$
\Bbb{F}_p\cong\pi_3(\widetilde{\rm{Rat}}_d)\otimes\Bbb{F}_p\rightarrow H_3(\widetilde{\rm{Rat}}_d;\Bbb{F}_p)
$$
is an isomorphism.\footnote{We do not need to define  homotopy groups with coefficients here since, as is proved in \cite[\S2]{MR2739026}, there is a universal coefficient sequence which implies $\pi_3(X;\Bbb{F}_p)\cong \pi_3(X)\otimes \Bbb{F}_p$ once $X$ is simply connected and $\pi_2(X)\otimes\Bbb{F}_p=0$.} Now consider the commutative diagram below which is the mod $p$ version of diagram \eqref{diagram}:
\begin{equation}\label{diagram'}
\xymatrixcolsep{5pc}\xymatrix{\pi_3({\rm{SL}}_2(\Bbb{C}))\otimes\Bbb{F}_p \ar[d]^{\sim}_{\mod p\text{ Hurewicz}}\ar[r] 
& \pi_3(\widetilde{\rm{Rat}}_d)\otimes\Bbb{F}_p \ar[d]^{\mod p\text{ Hurewicz}}_{\sim}\\
H_3({\rm{SL}}_2(\Bbb{C});\Bbb{F}_p) \ar[d]^{\sim} \ar[r]& H_3(\widetilde{\rm{Rat}}_d;\Bbb{F}_p) \ar@{->>}[d]\\
H_3({\rm{PSL}}_2(\Bbb{C});\Bbb{F}_p)\ar[r]& H_3({\rm{Rat}}_d;\Bbb{F}_p)\ar@/_/[u]_{\text{transfer}}
}
\end{equation} 
We know that all $\Bbb{F}_p$-vector spaces appearing here are one-dimensional except for $H_3({\rm{Rat}}_d;\Bbb{F}_p)$. This space is one-dimensional as well due to the epimorphism above and the fact that $H_3({\rm{Rat}}_d;\Bbb{F}_p)$ is non-trivial because $b_3({\rm{Rat}}_d)>0$ according to Proposition \ref{Rat homology}. The diagram indicates that, in order to establish the injectivity (and thus bijectivity) of homomorphisms \eqref{auxiliary15}, it suffices to show that the orbit inclusions for either conjugation or post-composition actions give rise to isomorphisms  
$$
\Bbb{F}_p\cong\pi_3({\rm{PSL}}_2(\Bbb{C}))\otimes\Bbb{F}_p\rightarrow\pi_3({\rm{Rat}}_d)\otimes\Bbb{F}_p\cong\Bbb{F}_p.
$$
To this end, we need to recall calculations carried out in the proof of Lemma \ref{main lemma}. Picking an arbitrary $f_0\in{\rm{Rat}}_d$, we proved that the homomorphism
\begin{equation}\label{auxiliary16}
\stackrel{\cong\Bbb{Z}}{\pi_3({\rm{PSL}}_2(\Bbb{C}))}\rightarrow
\stackrel{\cong \Bbb{Z}/2\Bbb{Z}\oplus\Bbb{Z}}{\pi_3({\rm{Rat}}_d)}\rightarrow
\stackrel{\cong\Bbb{Z}}{\pi_3(\Bbb{CP}^1)}
\end{equation}
induced by
$$
{\rm{PSL}}_2(\Bbb{C})\xrightarrow{\text{inclusion of the conjugacy class of }f_0}{\rm{Rat}}_d\xrightarrow{\text{evaluation at }\infty}\Bbb{CP}^1
$$
is the multiplication map by $1-d^2$ (see \eqref{auxiliary8}). However, the prime divisors of $2(d^2-1)$ have been ruled out as the characteristic and this means tensoring \eqref{auxiliary16} with $\Bbb{F}_p$ yields an isomorphism. Establishing that \eqref{auxiliary16} is an isomorphism in the case of the post-composition action is even easier: \eqref{auxiliary16} would correspond to the  evaluation map  
$$\begin{cases}
{\rm{PSL}}_2(\Bbb{C})\rightarrow\Bbb{CP}^1\\
\alpha\mapsto\alpha\left(f_0(\infty)\right)
\end{cases}
$$
which induces an isomorphism on $\pi_3$.\\
\indent
At this point, we have established the desired isomorphisms in \eqref{auxiliary14} at the level of graded $\Bbb{F}_p$-vector spaces. Nevertheless, the simple structure of the ring $H^*({\rm{PSL}}_2(\Bbb{C});\Bbb{F}_p)$ allows us to establish an algebra isomorphism between $H^*(\mathcal{M}_d;\Bbb{F}_p)$ and $H^*(\mathcal{M}^{\rm{post}}_d;\Bbb{F}_p)$ by an ad-hoc method. Denote the quotient maps for conjugation and post-composition actions by $$\pi:{\rm{Rat}}_d\rightarrow\mathcal{M}_d\text{, and }\pi^{\rm{post}}:{\rm{Rat}}_d\rightarrow\mathcal{M}^{\rm{post}}_d$$ respectively. To write down Leray-Hirsch isomorphisms, one needs to take the preimage of a generator $c$ of $H^3({\rm{PSL}}_2(\Bbb{C});\Bbb{F}_p)$ under the pullback maps $H^3({\rm{Rat}}_d;\Bbb{F}_p)\stackrel{\sim}{\rightarrow} H^3({\rm{PSL}}_2(\Bbb{C});\Bbb{F}_p)$ induced by different orbit inclusions. Here, all vector spaces involved are one-dimensional and so these preimages -- denoted by $a,a^{\rm{post}}$ -- are scalar multiples of each other. The isomorphisms \eqref{auxiliary14} of graded $\Bbb{F}_p$-vector spaces may now be described as 
\begin{equation*}
\begin{split}
&\begin{cases}
H^*(\mathcal{M}_d;\Bbb{F}_p)\otimes(\Bbb{F}_p.1\oplus\Bbb{F}_p.c)\rightarrow H^*({\rm{Rat}}_d;\Bbb{F}_p)\\
b\otimes 1\mapsto\pi^*(b),\, b\otimes c\mapsto\pi^*(b)\smile a; 
\end{cases}\\
&\begin{cases}
H^*(\mathcal{M}^{\rm{post}}_d;\Bbb{F}_p)\otimes(\Bbb{F}_p.1\oplus\Bbb{F}_p.c)\rightarrow H^*({\rm{Rat}}_d;\Bbb{F}_p)\\
b\otimes 1\mapsto\pi^{{\rm{post}}\,*}(b),\, b\otimes c\mapsto\pi^{{\rm{post}}\,*}(b)\smile a^{\rm{post}}. 
\end{cases}
\end{split}
\end{equation*}
Taking quotients modulo the subspace $H^*(\mathcal{M}^{\rm{post}}_d;\Bbb{F}_p)\otimes c$ on the left and the two-sided ideal 
$${\rm{I}}:=H^*({\rm{Rat}}_d;\Bbb{F}_p)\smile a=H^*({\rm{Rat}}_d;\Bbb{F}_p)\smile a^{\rm{post}}$$
of $H^*({\rm{Rat}}_d;\Bbb{F}_p)$ on the right yields the following graded $\Bbb{F}_p$-linear isomorphisms
\begin{equation*}
\begin{split}
&\begin{cases}
H^*(\mathcal{M}_d;\Bbb{F}_p)\rightarrow H^*({\rm{Rat}}_d;\Bbb{F}_p)\big/{\rm{I}}\\
b\mapsto\pi^*(b)\mod {\rm{I}}; 
\end{cases}\\
&\begin{cases}
H^*(\mathcal{M}^{\rm{post}}_d;\Bbb{F}_p)\rightarrow H^*({\rm{Rat}}_d;\Bbb{F}_p)\big/{\rm{I}}\\
b\mapsto\pi^{{\rm{post}}\,*}(b)\mod {\rm{I}}.
\end{cases}
\end{split}
\end{equation*}
Since these clearly preserve the ring structure (as any pullback map does), we conclude that both $\Bbb{F}_p$-algebras $H^*(\mathcal{M}_d;\Bbb{F}_p)$ and $H^*(\mathcal{M}^{\rm{post}}_d;\Bbb{F}_p)$ are isomorphic to $H^*({\rm{Rat}}_d;\Bbb{F}_p)\big/{\rm{I}}$ and thus isomorphic to each other. Finally, the fact that $\mathcal{M}^{\rm{post}}_d$ is homeomorphic to $\mathcal{T}_d$ (Theorem \ref{Toeplitz}) completes the proof.
\end{proof}

\begin{remark}
As mentioned before, $\mathcal{M}^{\rm{post}}_d$ is homotopy equivalent to ${\rm{Rat}}^*_d/S^1$. In this guise, the second isomorphism in $\eqref{auxiliary14}$ involving $H^*({\rm{Rat}}_d;\Bbb{F}_p)$ and  $H^*(\mathcal{M}^{\rm{post}}_d;\Bbb{F}_p)\cong H^*({\rm{Rat}}^*_d/S^1;\Bbb{F}_p)$   is also found in \cite[Proposition 1.2(a)]{MR1318152} using different notations and as an algebra isomorphism. 
\end{remark}

\begin{example}
The moduli space $\mathcal{M}_2$ is contractible \cite{MR1246482} and Theorem \ref{main finite} therefore implies that $\mathcal{T}_2$ is $\Bbb{F}_p$-acyclic for any prime characteristic $p\neq 2,3$. This can be confirmed with a direct calculation: $\mathcal{T}_2$ is the complement of the smooth quadric 
$$
C:=\left\{[a_1:a_2:a_3]\,\bigg|\, 
\det\begin{bmatrix}
a_2&a_3\\
a_1&a_2
\end{bmatrix}=0
\right\}
$$
in the projective plane $\Bbb{CP}^2$. By  the Poincar\'e-Lefschetz Duality 
$$
H_i(\mathcal{T}_2;\Bbb{F}_p)\cong H^{4-i}(\Bbb{CP}^2,C;\Bbb{F}_p);
$$
and the long exact sequence of cohomology groups for the pair $(\Bbb{CP}^2,C)$ then indicates that $H_i(\mathcal{T}_2;\Bbb{F}_p)$ is trivial for all $i>0$ unless $p=2$ in which case it is one-dimensional for $i=1,2$. 
\end{example}

\begin{corollary}\label{not contractible}
For any $d\geq 7$,  there is a prime  $p$ for which $H^{2p-2}(\mathcal{M}_d;\Bbb{F}_p)$ is non-trivial. 
\end{corollary}

\begin{proof}
Invoking Theorem \ref{main finite}, it suffices to show that $H^{2p-2}(\mathcal{T}_d;\Bbb{F}_p)\neq 0$ for an appropriate prime $p$ that does not divide any of the integers $d-1$, $d$ or $d+1$. To see this, we use the work of Milgram on the mod $p$ cohomology of the space $\mathcal{T}_d\cong{\rm{Rat}}^*_d/S^1$ of Toeplitz matrices. Taking  $\Bbb{F}_p$  as the coefficient group with $p$ an odd prime not dividing $d$, the  Serre spectral sequence for the fibration \eqref{fibration6} -- which converges to  $H^*(\mathcal{T}_d;\Bbb{F}_p)$ -- degenerates at the $E_3$ page and results in 
\begin{equation}\label{Milgram}
H^*({\rm{Rat}}^*_d;\Bbb{F}_p)\cong \Lambda_{\Bbb{Z}}[e_1]\otimes_{\Bbb{F}_p} E_\infty
\end{equation}
where $E_\infty=E_3$   and $\Lambda_{\Bbb{Z}}[e_1]$ is the exterior algebra generated by a generator $e_1$ of 
$H^1({\rm{Rat}}^*_d;\Bbb{Z})\cong\Bbb{Z}$ \cite[Theorem 2.2]{MR1445557}.\footnote{Recall that $\pi_1({\rm{Rat}}^*_d)\cong\Bbb{Z}$; cf. Proposition \ref{Rat fundamental}. There is also a typo in the paper: in the statement of the theorem $\Bbb{Z}/2\Bbb{Z}$ must be replaced with $\Bbb{Z}/p\Bbb{Z}$; the correct statement can be found in \cite[Theorem 4.1]{MR1318152}.} The left hand side of \eqref{Milgram} is known: ${\rm{Rat}}^*_d$ is of the stable homotopy type of the Braid group ${\rm{B}}_{2d}$ on $2d$ strands \cite[Theorem 1.1]{MR1097023} whose  cohomology groups $H^*({\rm{B}}_{2d};\Bbb{F}_p)$ have been computed. For instance, for $p\leq d$  an odd prime, the first non-zero mod $p$ cohomology of the Braid group ${\rm{B}}_{2d}$ in dimensions larger than one happens in dimension $2p-2$ \cite[chap. III, Appendix]{MR0436146}. Now, by comparing the dimensions of degree $2p-2$ pieces of the different sides of  \eqref{Milgram} we deduce that $H^{2p-2}(\mathcal{T}_d;\Bbb{F}_p)$ is non-trivial. Consequently, to finish the proof we just need to show that if $d\geq 7$, there always exists a prime $2<p<d$ not dividing any of  $d-1$, $d$ or $d+1$. The gcd of each of these numbers with either $d-2$ or $d-3$ is at most $4$. Hence, any prime factor other than $2$ or $3$ of either $d-2$ or $d-3$ works as $p$. If there is no such a prime factor, both $d-2$ and $d-3$ must be in the form of $2^m3^n$. However, they differ by one, so 
$\{d-3,d-2\}=\left\{2^a,3^b\right\}$ where $2^a-3^b=\pm 1$. It is not hard to check that 
$$(a,b)=(1,0),(1,1),(2,1),(3,2)$$
are the only solutions of this Diophantine equation. They correspond to $d=4,5,6,11$ respectively. Only the last one matters here for which the existence of such a $p$ can be verified directly: $p=7$ is an odd prime less than $d=11$ that divides neither of $d-1$, $d$ and $d+1$.
\end{proof}

\begin{proof}[Proof of Theorem \ref{main 2.5}]
Follows immediately from Corollary \ref{not contractible}.
\end{proof}

\section{The fundamental group of \texorpdfstring{$\mathcal{M}_d$}{Md}}\label{homotopy groups}
The key ingredient in computing $\pi_1(\mathcal{M}_d)$ is the following proposition:
\begin{proposition}[A special case of {\cite[chap. I, Proposition 8.10]{MR1027600}}]\label{technical}
If $G$ is a Lie group acting properly on a simply connected manifold $X$, then there is an isomorphism 
$$
\pi_1(X/G)\cong \pi_0(G)/M
$$
where $M$ is the normal subgroup of $\pi_0(G)$ generated by $\left\{\pi_0({\rm{Stab}}(x))\right\}_{x\in X}$. 
\end{proposition}

\begin{remark}
Proposition \ref{technical} is a special case of  {\cite[chap. I, Proposition 8.10]{MR1027600}}. 
Assuming that the space\footnote{The space $X$ is not required to be a manifold in \cite{MR1027600}. It only needs to be connected, locally path connected,
completely regular and possessing a universal cover.} $X$ is connected and the action of $G$ on $X$ is proper, the aforementioned reference exhibits an exact sequence 
$$\pi_1(X,x_0)\rightarrow\pi_1(X/G,x_0G)\rightarrow\pi_0(G)/M\rightarrow 0$$
and describes the kernel of the first homomorphism. The last homomorphism $$\pi_1(X/G,x_0G)\rightarrow\pi_0(G)/M$$ takes the homotopy class of a loop $\gamma$ based at $x_0G$ to the class represented by $g\in G$ whenever $\gamma$ admits a lift to a path $\tilde{\gamma}$ in $X$ with $\tilde{\gamma}(0)=x_0$ and $\tilde{\gamma}(1)=g.x_0$. 
As another special case pertinent to this article, for $G$ connected the map  $\pi_1(X)\rightarrow\pi_1(X/G)$ induced by the quotient map is an epimorphism. This can also be deduced from the main result of \cite{MR2852978} concerning quotient maps with connected fibers.
\end{remark}

Recall that a continuous action $G\curvearrowright X$ of a topological group $G$ on a topological space $X$ is called \textit{proper} if the map 
$$
\begin{cases}
G\times X\rightarrow X\times X\\
(g,x)\mapsto(g\cdot x,x)
\end{cases}
$$
is proper. This condition guarantees that the orbit space $X/G$ is not pathological \cite[chap. I, Lemma 1.19]{MR1027600}. 
Hence, in order to invoke Proposition \ref{technical}, we shall first prove the properness of the conjugation action.
This is the purpose of the following lemma whose proof is of dynamical flavour.
\begin{lemma}\label{proper}
The conjugation action ${\rm{PSL}}_2(\Bbb{C})\curvearrowright{\rm{Rat}}_d$ \eqref{auxiliary1} is proper.
\end{lemma} 

\begin{proof}
We need to establish that given a sequence $\left\{\alpha_n\right\}_n$ of M\"obius transformations and a sequence $\left\{f_n\right\}_n$ of degree $d$ rational maps converging to $f\in{\rm{Rat}}_d$, if  $\left\{\alpha_n\circ f_n\circ\alpha_n^{-1}\right\}_n$ also admits a limit in ${\rm{Rat}}_d$, say $g$,  then the sequence $\left\{\alpha_n\right\}_n$ of elements of ${\rm{PSL}}_2(\Bbb{C})$ has a convergent subsequence. We use the following well-known criterion for the convergence of a sequence $\left\{\alpha_n\right\}_n$ of M\"obius transformations:
\begin{itemize}
\item \textit{Let $\left\{p_n\right\}_n$, $\left\{q_n\right\}_n$, and $\left\{r_n\right\}_n$  be convergent sequences of points on the Riemann sphere  such that the transformed  sequences $\left\{\alpha_n(p_n)\right\}_n$, $\left\{\alpha_n(q_n)\right\}_n$ and $\left\{\alpha_n(r_n)\right\}_n$ converge as well:
\begin{equation*}
\begin{split}
&p_n\to p,\quad q_n\to q,\quad r_n\to r;\\
&\alpha_n(p_n)\to p',\quad \alpha_n(q_n)\to q',\quad \alpha_n(r_n)\to r'.
\end{split}
\end{equation*}
If the triples  $p,q,r$ and also $p',q',r'$ are pairwise distinct, then  $\left\{\alpha_n\right\}_n$ converges in ${\rm{PSL}}_2(\Bbb{C})$.} 
\end{itemize}
\indent
Recall that periodic points map to periodic points of the same (exact) period under the conjugation action. Since $f_n\to f$ (respectively $\alpha_n\circ f_n\circ\alpha_n^{-1}\to g$), any limit point of period $k$ points of elements of  $\left\{f_n\right\}_n$ (resp. of elements of $\left\{\alpha_n\circ f_n\circ\alpha_n^{-1}\right\}_n$) is a periodic point of $f$ (resp. of $g$). The idea is to use the fact that points of different periods are distinct and then apply the criterion above. The issue is that  period $k$ points of maps $f_n$ (or of  maps $\alpha_n\circ f_n\circ\alpha_n^{-1}$) may accumulate to a periodic point of $f$ (resp. of $g$) whose period is a proper divisor of $k$. This means a \textit{bifurcation} of periodic points has taken place; a situation that must be avoided. \\
\indent
Consider a periodic point $z_0$ of $f$ whose period is $k$. Such a point is a solution to $f^{\circ k}(z)-z=0$, but not every solution of this equation is of exact period $k$; periods which are proper divisors of $k$ may occur too. Perturbing $f$ to another map and $z_0$ to another solution, the period of the solution as a periodic point of the nearby map may increase. Such a difficulty could be circumvented if the \textit{multipliers} $\left(f^{\circ k}(z_0)\right)'$ of the solutions $z_0$ of $f^{\circ k}(z)-z=0$ differ from $1$. In that case, by the fundamental theorem of algebra, $f$ admits $d^k+1$ distinct periodic points of periods $m|k$ on the Riemann sphere, each of which could be analytically continued to periodic points of the same period for the nearby maps by an application of the implicit function theorem. When the multiplier is $1$ the periodic point is \textit{parabolic}. It is a standard fact from holomorphic dynamics that there are only finitely many parabolic cycles \cite[Corollary 10.16]{MR2193309}. Therefore, we can take distinct positive integers $k_1,k_2$ and $k_3$ so large that no periodic point of $f$ or $g$ with these periods is parabolic. As such, each periodic point of $f$ or $g$ with one of these periods can be analytically continued to a curve of points of the same period for maps close enough to $f$ or $g$ in ${\rm{Rat}}_d$. We deduce that, fixing a period $k_1$, $k_2$ or $k_3$, as $n\to\infty$ points of that period under $f_n$ (respectively, under $\alpha_n\circ f_n\circ\alpha_n^{-1}$) accumulate to the same kind of points for $f$ (resp. for $g$). Now, choosing points $p$, $q$ and $r$ whose periods under $f$ are $k_1$, $k_2$ and $k_3$ respectively, there are sequences $\{p_n\}_n$, $\{q_n\}_n$ and $\{r_n\}_n$ of points of the same periods under maps $f_n$ with $p_n\to p$, $q_n\to q$ and $r_n\to r$. Conjugating by $\alpha_n$, we obtain sequences  $\left\{\alpha_n(p_n)\right\}_n$, $\left\{\alpha_n(q_n)\right\}_n$ and $\left\{\alpha_n(r_n)\right\}_n$ of periodic points of maps $\alpha_n\circ f_n\circ\alpha_n^{-1}$ with the same description of periods. Passing to subsequences if necessary, the limits $p'$, $q'$ and $r'$ of these latter sequences must be periodic under $g$ with  periods $k_1$, $k_2$ and $k_3$ respectively. The points $p,q,r$ and also $p',q',r'$ are pairwise distinct as they are of different periods. The aforementioned criterion now finishes the proof. 
\end{proof}

Proposition \ref{technical} must be applied to an action on a simply connected space. This is why we consider a group action on the universal cover $\widetilde{\rm{Rat}}_d$. The space is simultaneously equipped with the covering space action \eqref{action1} of the group of $2d^{\rm{th}}$ roots of unity and the action \eqref{action2} of ${\rm{SL}}_2(\Bbb{C})$. These actions clearly commute, so we obtain the following commutative diagram of group actions: 
\begin{equation}\label{diagram1}
\xymatrixcolsep{6pc}
\xymatrix{\widetilde{\rm{Rat}}_d\ar[r]^{{\rm{SL}}_2(\Bbb{C})\curvearrowright}\ar[d]_{\Bbb{Z}/2d\Bbb{Z}\curvearrowright}
&\widetilde{\rm{Rat}}_d\big/{\rm{SL}}_2(\Bbb{C})\ar[d]_{\Bbb{Z}/2d\Bbb{Z}\curvearrowright}\\
{\rm{Rat}}_d\ar[r]^{{\rm{PSL}}_2(\Bbb{C})\curvearrowright}&\mathcal{M}_d
}
\end{equation} 

\begin{proof}[Proof of Theorem \ref{main 2}]
In view of diagram \eqref{diagram1}, we prove the theorem by applying Proposition \ref{technical} to a left action of $\{\lambda\mid\lambda^{2d}=1\}\times{\rm{SL}}_2(\Bbb{C})$ on  $\widetilde{{\rm{Rat}}}_d$ whose components are given by \eqref{action1} and \eqref{action2}:

\begin{equation}\label{action3}
\left(\lambda,
\begin{bmatrix}
r&s\\
t&u
\end{bmatrix}\right)\cdot
\left(P(X,Y),Q(X,Y)\right)=(P'(X,Y),Q'(X,Y))
\end{equation}
where
$$
P'(X,Y):=\lambda rP(uX-sY,-tX+rY)+\lambda sQ(uX-sY,-tX+rY)\text{,}
$$
and
$$
Q'(X,Y):=\lambda tP(uX-sY,-tX+rY)+\lambda uQ(uX-sY,-tX+rY).
$$
The action is clearly proper since it is a lift of the conjugation action of ${\rm{PSL}}_2(\Bbb{C})$ on ${\rm{Rat}}_d$ to an action of the finite-sheeted cover $\{\lambda\mid\lambda^{2d}=1\}\times{\rm{SL}}_2(\Bbb{C})$ of ${\rm{PSL}}_2(\Bbb{C})$. 
The group of connected components of the former group can be identified with $\{\lambda\mid\lambda^{2d}=1\}$ and, in order to apply Proposition \ref{technical}, we must show that the subgroup of it generated by 
the first component of elements $(\lambda,A)$ that fix a point in action \eqref{action3} coincides with $\{\lambda\mid\lambda^{2d}=1\}$ if $d$ is even and is an index two subgroup of it if $d$ is odd.  Denote the generator ${\rm{e}}^{\frac{\pi{\rm{i}}}{d}}$ of the group $\{\lambda\mid\lambda^{2d}=1\}$ by $\omega$. We have:
\begin{equation}\label{auxiliary17}
\begin{split}
&\left(-\omega^{-1},\begin{bmatrix}
\omega&0\\
0&\omega^{-1}
\end{bmatrix}\right)\cdot\left(X^d+Y^d,X^{d-1}Y\right)\\
&=-\omega^{-1}\left(\omega(\omega^{-1}X)^d+\omega(\omega Y)^d,\omega^{-1}(\omega^{-1}X)^{d-1}(\omega Y)\right)
=\left(-\omega^{-d}X^d-\omega^d Y^d,-\omega^{-d}X^{d-1}Y\right)\\
&=\left(X^d+Y^d,X^{d-1}Y\right);
\end{split}
\end{equation}
\normalsize
that is, an element from the connected component $\left\{-\omega^{-1}\right\}\times{\rm{SL}}_2(\Bbb{C})$ belongs to a stabilizer. 
However, $$-\omega^{-1}=-{\rm{e}}^{-\frac{\pi{\rm{i}}}{d}}={\rm{e}}^{\frac{\pi{\rm{i}}(d-1)}{d}}$$ is a primitive $2d^{\rm{th}}$ root of unity if and only if $\gcd(2d,d-1)=1$; i.e. $d$ is even. For $d$ odd, it would be of order $\frac{2d}{\gcd(2d,d-1)}=d$ and hence generate a subgroup of index two. So far, we have shown that $\pi_1(\mathcal{M}_d)$ is trivial for $d$ even while is of order at most two for $d$ odd. To prove the equality in the latter case, one needs to verify that if $d\geq 3$ is odd, $\omega$ itself cannot occur as the first component of a stabilizer of the action \eqref{action3}. Assuming the contrary, there is a unimodular matrix $A$ and homogeneous coprime polynomials $P(X,Y),Q(X,Y)$ of degree $d$ satisfying 
$$
(\omega,A)\cdot\left(P(X,Y),Q(X,Y)\right)=\left(P(X,Y),Q(X,Y)\right).
$$
This implies that the M\"obius transformation that $A$ descends to is an automorphism of the degree $d$ rational map that $P$ and $Q$ determine. In particular, $A$ must be of finite order and hence diagonalizable. Replacing $\left(P(X,Y),Q(X,Y)\right)$ with a suitable element of its orbit, we may assume that $A$ is already diagonal with non-zero entries $a$ and $a^{-1}$ in its first and second row. 
 Substituting into the equation above, we get:
$$
\begin{cases}
aP\left(a^{-1}X,aY\right)=\omega^{-1} P(X,Y),\\
a^{-1}Q\left(a^{-1}X,aY\right)=\omega^{-1}Q(X,Y).
\end{cases}
$$
Comparing the coefficients of $X^d$ and $Y^d$ on different sides, we deduce that $\omega^{-1}$ is equal to $a^{-(d-1)}$ or $a^{-(d+1)}$ if $\deg_XP=d$ or $\deg_XQ=d$ respectively; while it is equal to $a^{d+1}$ or $a^{d-1}$  if $\deg_YP=d$ or $\deg_YQ=d$ respectively. In each situation at least one of the possibilities takes place since $\max\{\deg_XP,\deg_XQ\}=d$ and 
$\max\{\deg_YP,\deg_YQ\}=d$ due to the fact that $P$ and $Q$ are coprime. Thus, there are two expressions $a^{d+\epsilon_1}$ and $a^{-(d+\epsilon_2)}$ of $\omega^{-1}={\rm{e}}^{-\frac{\pi{\rm{i}}}{d}}$ where $\epsilon_1,\epsilon_2\in\{\pm 1\}$. If $\epsilon_1$ and $\epsilon_2$ coincide then $\omega^{-1}=\pm 1$; a contradiction. Otherwise, 
$a^{d+\epsilon_1}=a^{-(d+\epsilon_2)}$ amounts to $a^{2d}=1$. Hence, the generator $\omega^{-1}$ of $\{\lambda\mid\lambda^{2d}=1\}$ is the $(d+1)^{\rm{th}}$ or the $(d-1)^{\rm{th}}$ power of another element of this cyclic group. This is impossible because $\gcd(2d,d\pm 1)=2$.
\end{proof}

We conclude this section with an investigation of the fundamental groups of a number of related spaces. Proposition \ref{technical} can be used to compute the fundamental group of the space $\mathcal{M}_d^{\rm{pre}}$ introduced in \S\ref{application variant}. 
Calculating the fundamental group of the other quotient space $\mathcal{M}_d^{\rm{post}}$ that appeared in that section is easier since the post-composition action of ${\rm{PSL}}_2(\Bbb{C})$ on 
${\rm{Rat}}_d$ is free: its fundamental group is isomorphic to $\Bbb{Z}/d\Bbb{Z}$  \cite[Lemma 3.4]{MR1318152}. In the proposition below we focus on $\mathcal{M}_d^{\rm{pre}}$.
\begin{proposition}\label{variant fundamental group}
For any $d\geq 2$ one has $\pi_1(\mathcal{M}_d^{\rm{pre}})\cong\Bbb{Z}/d\Bbb{Z}$.  
\end{proposition}

\begin{proof}
Instead of the action \eqref{action3} of $\{\lambda\mid\lambda^{2d}=1\}\times{\rm{SL}}_2(\Bbb{C})$ on $\widetilde{{\rm{Rat}}}_d$, we consider another left action in which ${\rm{SL}}_2(\Bbb{C})$
acts only by pre-composition as in \eqref{action2 pre-variant}: 
\small
\begin{equation}\label{action4}
\left(\lambda,
\begin{bmatrix}
r&s\\
t&u
\end{bmatrix}\right)\cdot
\left(P(X,Y),Q(X,Y)\right):=
\left(\lambda P(uX-sY,-tX+rY),\lambda Q(uX-sY,-tX+rY)\right).
\end{equation}
\normalsize
The properness of this action (or equivalently that of the pre-composition action of ${\rm{PSL}}_2(\Bbb{C})$ on ${\rm{Rat}}_d$) follows at once from the criterion mentioned in the proof of Lemma \ref{proper}. Imitating the same arguments, we need to characterize the subgroup of $\{\lambda\mid\lambda^{2d}=1\}$ generated by elements that occur as the first component of a member of a stabilizer of action \eqref{action4}. The element $\lambda=-1$ certainly qualifies since 
\small
$$
\left(-1,
\begin{bmatrix}
\omega&0\\
0&\omega^{-1}
\end{bmatrix}\right)\cdot\left(X^d,Y^d\right)=-\left((\omega^{-1}X)^d,(\omega Y)^d\right)
=\left(-\omega^{-d}X^d,-\omega^d Y^d\right)=\left(X^d,Y^d\right)
$$
\normalsize
where $\omega:={\rm{e}}^{\frac{\pi{\rm{i}}}{d}}$.
We claim that $\pm 1$ are the sole $2d^{\rm{th}}$ roots of unity that occur in an element $(\lambda,A)$ stabilizing an element $\left(P(X,Y),Q(X,Y)\right)$. Again, it is no loss of generality to take $A\in{\rm{SL}}_2(\Bbb{C})$ to be a diagonal matrix with non-zero entries $a$ and $a^{-1}$ in its first and second row.
We now obtain identities 
$P\left(a^{-1}X,aY\right)=\lambda^{-1}P(X,Y)$ and $Q\left(a^{-1}X,aY\right)=\lambda^{-1}Q(X,Y)$. Either $\deg_XP$ or $\deg_XQ$ is $d$, and thus $\lambda^{-1}=a^{-d}$ by comparing the coefficients of $X^d$. The same reasoning based on the coefficients of $Y^d$ yields $\lambda^{-1}=a^d$. Therefore, $\lambda=\lambda^{-1}$ which implies $\lambda=\pm1$. 
\end{proof}

We now exhibit an explicit generator for the fundamental groups.

\begin{proposition}\label{generator}
The 1-parameter family
\begin{equation}\label{loop}
\left\{z\mapsto z^{d-1}+\frac{{\rm{e}}^{2\pi{\rm{i}}t}}{z}\right\}_{t\in [0,1]}    
\end{equation}
determines generators for $\pi_1({\rm{Rat}}^*_d)$ and $\pi_1({\rm{Rat}}_d)$, and also for $\pi_1(\mathcal{M}_d)$ (of course only interesting for $d$ odd). 
\end{proposition}

\begin{proof}
The homomorphism $\pi_1({\rm{Rat}}^*_d)\rightarrow\pi_1({\rm{Rat}}_d)$ induced by inclusion is a surjection (Proposition \ref{Rat fundamental}) and the same is true for the homomorphism 
$\pi_1({\rm{Rat}}_d)\rightarrow\pi_1(\mathcal{M}_d)$ induced by the quotient map \cite{MR2852978}. Next, thinking of ${\rm{Rat}}^*_d$ as the space of rational maps of degree $d$ with $\infty\mapsto \infty$, we show that the loop \eqref{loop} gives rise to a generator of its fundamental group. By Proposition \ref{Rat fundamental}, it suffices to show that 
$$t\mapsto\frac{{\rm{Res}}\left(z^d+{\rm{e}}^{2\pi{\rm{i}}t},z\right)}{{\rm{e}}^{2\pi{\rm{i}}t}}$$
is constant where {\rm{Res}} denotes the degree $d$ resultant.\footnote{In the proposition, ${\rm{Rat}}^*_d$ is thought of as the subspace of maps with $\infty\mapsto 1$ while here we are working with maps taking $\infty$ to $\infty$. The resultant map \eqref{resultant map} must be considered as $\frac{p}{q}\in{\rm{Rat}}^*_d\mapsto{\rm{Res}}(p,q)\in\Bbb{C}^*$ where $p$ is monic of degree $d$ and $\deg q<d$. Clearly, it can still be used to identify the fundamental group of ${\rm{Rat}}^*_d$ with $\Bbb{Z}$.} One can proceed by computing this resultant, but a more elegant approach is to show that
${\rm{Res}}\left(z^d+c,z\right)$ with $c\neq 0$ is linear in $c$. Pre-composing the polynomials with $z\mapsto c^{\frac{1}{d}}z$  modifies the resultant as follows (cf. \cite[Exercise 2.7]{MR2316407})
$$\left(c^{\frac{1}{d}}\right)^{-d^2}{\rm{Res}}\left(cz^d+c,c^{\frac{1}{d}}z\right).$$ Hence,
\begin{equation*}
\begin{split}
{\rm{Res}}\left(z^d+c,z\right)&=\left(c^{\frac{1}{d}}\right)^{-d^2}{\rm{Res}}\left(c(z^d+1),c^{\frac{1}{d}}z\right)\\&=\left(c^{\frac{1}{d}}\right)^{-d^2}c^d\left(c^{\frac{1}{d}}\right)^d
{\rm{Res}}\left(z^d+1,z\right)=c\,{\rm{Res}}\left(z^d+1,z\right).   
\end{split}
\end{equation*}
\end{proof}

Our treatment of $\pi_1(\mathcal{M}_d)$ in this section relied on a result such as Proposition \ref{technical} because the conjugation action of ${\rm{PSL}}_2(\Bbb{C})$ on ${\rm{Rat}}_d$ is not free. The locus of non-free orbits is the symmetry locus $\mathcal{S}$ which is of the relatively large codimension $d-1$, and is the same as the singular locus of $\mathcal{M}_d$ when $d\geq 3$ (cf. Proposition \ref{symmetry locus}). Removing it, we arrive at a smooth variety $\mathcal{M}_d-\mathcal{S}$ which fits in the fiber bundle
\begin{equation}\label{fibration7}
{\rm{PSL}}_2(\Bbb{C})\hookrightarrow{\rm{Rat}}_d-\mathbf{S}\rightarrow\mathcal{M}_d-\mathcal{S}     
\end{equation}
where $\mathbf{S}$ is the preimage of $\mathcal{S}$ in ${\rm{Rat}}_d$. 
The final proposition of this section utilizes this fibration to compute the fundamental group of the smooth locus $\mathcal{M}_d-\mathcal{S}$.

\begin{proposition}\label{the smooth locus}
Given any $d\geq 3$, one has 
$$
\pi_1(\mathcal{M}_d-\mathcal{S})=
\begin{cases}
\Bbb{Z}/2d\Bbb{Z}\quad d\text{ odd}\text{, and}\\
\Bbb{Z}/d\Bbb{Z}\,\,\quad d\text{ even}.
\end{cases}
$$
\end{proposition}

\begin{proof}
We need to analyze the homomorphism $\pi_1({\rm{PSL}}_2(\Bbb{C}))\rightarrow\pi_1({\rm{Rat}}_d-\mathbf{S})$ from the long exact sequence of homotopy groups corresponding to \eqref{fibration7}. The real codimension of the closed subset $\mathbf{S}$ of the connected manifold ${\rm{Rat}}_d$ is $2(d-1)>2$; thus, its removal does not affect the fundamental group. As such, we need to show that the homomorphism 
$$
\Bbb{Z}/2\Bbb{Z}\cong\pi_1({\rm{PSL}}_2(\Bbb{C}))\rightarrow\pi_1({\rm{Rat}}_d)\cong\Bbb{Z}/2d\Bbb{Z}
$$
induced by the inclusion of an orbit of the conjugation action is trivial when $d$ is odd and an injection for $d$ even. In view of Proposition \ref{generator}, we consider the orbit of $f(z)=z^{d-1}+\frac{1}{z}$. The loop $\left\{\alpha(t)={\rm{e}}^{2\pi{\rm{i}}t}z\right\}_{t\in [0,1]}$ of M\"obius transformations generates the fundamental group of ${\rm{PSL}}_2(\Bbb{C})$. Conjugating $f$ with these transformations, we arrive at the 1-parameter family
$$
\left\{\alpha(t)\circ f\circ\alpha(t)^{-1}:z\mapsto \frac{1}{\left({\rm{e}}^{2\pi{\rm{i}}t}\right)^{d-2}}\left(z^{d-1}+\frac{({\rm{e}}^{2\pi{\rm{i}}t})^d}{z}\right)\right\}_{t\in [0,1]}
$$
of rational maps of degree $d$ that take $\infty$ to $\infty$. Repeating the arguments in the proof of Proposition \ref{generator}, we find the ratio 
$$
\frac{{\rm{Res}}\left(z^d+({\rm{e}}^{2\pi{\rm{i}}t})^d,\left({\rm{e}}^{2\pi{\rm{i}}t}\right)^{d-2}z\right)}{{\rm{e}}^{2\pi{\rm{i}}t}}.
$$
As observed before, $t\mapsto{\rm{Res}}\left(z^d+{\rm{e}}^{2\pi{\rm{i}}t},z\right)$ is a non-zero multiple of $t\mapsto{\rm{e}}^{2\pi{\rm{i}}t}$. Hence, the numerator
$$
{\rm{Res}}\left(z^d+({\rm{e}}^{2\pi{\rm{i}}t})^d,\left({\rm{e}}^{2\pi{\rm{i}}t}\right)^{d-2}z\right)
=\left(\left({\rm{e}}^{2\pi{\rm{i}}t}\right)^{d-2}\right)^d{\rm{Res}}\left(z^d+({\rm{e}}^{2\pi{\rm{i}}t})^d,z\right)
$$
above is a non-zero multiple of $\left({\rm{e}}^{2\pi{\rm{i}}t}\right)^{d(d-2)+d}=\left({\rm{e}}^{2\pi{\rm{i}}t}\right)^{d^2-d}$. The element determined by the integer $d^2-d$ in the group  $\Bbb{Z}/2d\Bbb{Z}$ is trivial  for $d$ odd and is of order two when $d$ is even. 
\end{proof}

\begin{remark}
Notice that $\pi_1(\mathcal{M}_d-\mathcal{S})$ is not the same as $\pi_1(\mathcal{M}_d)$ even though the codimension of $\mathcal{S}$ is large.  
Indeed, for a normal irreducible complex variety $X$ ($\mathcal{M}_d$ is normal \cite[Theorem 2.1(c)]{MR1635900}), one can merely say that the homomorphism $\pi_1(X_{\rm{smooth}})\rightarrow\pi_1(X)$ is surjective \cite{MR3493218}.
\end{remark}

\section{The higher homotopy groups of \texorpdfstring{${\rm{Rat}}_d$}{Ratd}}\label{application resultant}
In this section, we prove Theorem \ref{main 3} after a sequence of preliminary results.

The commuting actions on the universal cover $\widetilde{\rm{Rat}}_d$ illustrated in diagram \eqref{diagram1} were essential to our calculation of $\pi_1(\mathcal{M}_d)$ in the previous section. Recall that the rows of the diagram were not fiber bundles because the conjugation action is not free. Working with the post-composition action introduced in \S\ref{application variant} instead, we obtain a similar commutative diagram of actions  
\begin{equation}\label{diagram2}
\xymatrixcolsep{6pc}
\xymatrix{\widetilde{\rm{Rat}}_d\ar[r]^{{\rm{SL}}_2(\Bbb{C})\stackrel{{\rm{post}}}{\curvearrowright}}\ar[d]_{\Bbb{Z}/2d\Bbb{Z}\curvearrowright}
&\widetilde{\rm{Rat}}_d\big/{\rm{SL}}_2(\Bbb{C})\ar[d]_{\Bbb{Z}/2d\Bbb{Z}\curvearrowright}\\
{\rm{Rat}}_d\ar[r]^{{\rm{PSL}}_2(\Bbb{C})\stackrel{{\rm{post}}}{\curvearrowright}}&\mathcal{M}_d^{\rm{post}}
}
\end{equation} 
which shall be used to prove Theorem \ref{main 3} in this section. The advantage of this new diagram is that the actions appearing in the rows are now free. Indeed, they are given by the post-composition action of M\"obius transformations on ${\rm{Rat}}_d$ and its lift (cf. \eqref{action2 post-variant})  to $\widetilde{\rm{Rat}}_d$. The left column is the covering space action \eqref{action1} of the group of $2d^{\rm{th}}$ roots of unity. In the right column, this group acts on the quotient $\widetilde{\rm{Rat}}_d\big/{\rm{SL}}_2(\Bbb{C})$ 
 because the two actions on $\widetilde{\rm{Rat}}_d$ commute. However, this latter action of the group of $2d^{\rm{th}}$ roots of unity is not free. 

\begin{lemma}\label{lemma3}
Let $d$ be an integer larger than one. The quotient $\widetilde{\rm{Rat}}_d\big/{\rm{SL}}_2(\Bbb{C})$ from diagram \eqref{diagram2} is simply connected. 
The action of the group of $2d^{\rm{th}}$ roots of unity in the right column of the diagram is free modulo a kernel generated by the root of unity $-1$. In particular, $\widetilde{\rm{Rat}}_d\big/{\rm{SL}}_2(\Bbb{C})$ is always the universal cover of $\mathcal{M}_d^{\rm{post}}$. 
\end{lemma}

\begin{proof}
The top row of diagram \eqref{diagram2} describes $\widetilde{\rm{Rat}}_d\big/{\rm{SL}}_2(\Bbb{C})$ as the base of a fiber bundle whose total space $\widetilde{\rm{Rat}}_d$ and fiber ${\rm{SL}}_2(\Bbb{C})$ are both simply connected. Hence, $\widetilde{\rm{Rat}}_d\big/{\rm{SL}}_2(\Bbb{C})$ must be simply connected as well. \\
\indent
Suppose a root of unity $\lambda\neq 1$ stabilizes the post-composition class of $(P,Q)\in\widetilde{{\rm{Rat}}}_d$. Namely, there is a matrix 
$$A=
\begin{bmatrix}
r&s\\
t&u
\end{bmatrix}\in{\rm{SL}}_2(\Bbb{C})$$
satisfying
$$
(\lambda P,\lambda Q)=A\cdot(P,Q)=(rP+sQ,tP+uQ). 
$$
This amounts to 
$$\begin{bmatrix}
r&s\\
t&u
\end{bmatrix}=\lambda{\rm{I}}_2$$
since, otherwise, the homogeneous polynomials $P(X,Y)$ and $Q(X,Y)$ would be scalar multiples of each other which is impossible as they are coprime. We deduce that $\lambda=-1$ since
$A$
is unimodular. Therefore, in the action from the right column of diagram \eqref{diagram2}, $\lambda=-1$ is the only non-trivial group element that can stabilize a point. As a matter of fact, it fixes any arbitrary element ${\rm{SL}}_2(\Bbb{C})\cdot(P,Q)\in\widetilde{\rm{Rat}}_d\big/{\rm{SL}}_2(\Bbb{C})$ since:
$$
\begin{bmatrix}
-1&0\\
0&-1
\end{bmatrix}\cdot(P,Q)
=(-P,-Q)=(-1)\cdot(P,Q).
$$
\indent So far, we have shown that  $\mathcal{M}_d^{\rm{post}}$ is the quotient of the simply connected space $\widetilde{\rm{Rat}}_d\big/{\rm{SL}}_2(\Bbb{C})$ by a free action of a cyclic group of order $d$. The latter is thus a universal cover of the former. 
\end{proof}

\begin{remark}
As mentioned in \S\ref{application variant}, the fundamental group of $\mathcal{M}_d^{\rm{post}}$ is known to be isomorphic to $\Bbb{Z}/d\Bbb{Z}$  \cite[Lemma 3.4]{MR1318152}. The lemma above verifies this by exhibiting an explicit universal cover.
\end{remark}

\begin{lemma}\label{lemma1}
Given any $d>3$, we have  
$$
\pi_i(\mathcal{M}_d^{\rm{post}})\otimes\Bbb{Q}\cong
\begin{cases}
0 & 1\leq i< d\text{, and}\\
\pi_i({\rm{Rat}}_d)\otimes\Bbb{Q}& i\geq d.
\end{cases}
$$
\end{lemma}

\begin{proof}
Consider the long exact sequence of  homotopy groups associated with the fiber bundle
$$
{\rm{PSL}}_2(\Bbb{C})\hookrightarrow{\rm{Rat}}_d\rightarrow\mathcal{M}_d^{\rm{post}}
$$
appearing in the bottom row of diagram \eqref{diagram2}. Its first few terms are 
$$
0=\pi_2({\rm{PSL}}_2(\Bbb{C}))\rightarrow\pi_2({\rm{Rat}}_d)\rightarrow\pi_2(\mathcal{M}_d^{\rm{post}})\rightarrow\stackrel{\cong\Bbb{Z}/2\Bbb{Z}}{\pi_1({\rm{PSL}}_2(\Bbb{C}))}
\rightarrow\pi_1({\rm{Rat}}_d)\rightarrow\pi_1(\mathcal{M}_d^{\rm{post}})\rightarrow 0.
$$
\normalsize
The groups $\pi_2({\rm{Rat}}_d)$ and $\pi_1({\rm{Rat}}_d)$ are finite and cyclic according to Theorem \ref{Rat homotopy} and Proposition \ref{Rat fundamental} respectively. We deduce that  
$\pi_1(\mathcal{M}_d^{\rm{post}})$ and $\pi_2(\mathcal{M}_d^{\rm{post}})$ are also finite and cyclic. Next, we need to address the rational homotopy groups $\pi_i(\mathcal{M}_d^{\rm{post}})\otimes\Bbb{Q}$ for $i\geq 3$. The group $\pi_i({\rm{PSL}}_2(\Bbb{C}))\cong\pi_i(\Bbb{RP}^3)$ is finite unless $i=3$ in which case it is an infinite cyclic group. Consequently, the long exact sequence of homotopy groups implies that
\begin{equation}\label{auxiliary1'}
\forall\, i\geq 5: \pi_i({\rm{Rat}}_d)\otimes\Bbb{Q}\cong\pi_i(\mathcal{M}_d^{\rm{post}})\otimes\Bbb{Q}.
\end{equation}
When $3\leq i<d$, Theorem \ref{Rat homotopy} describes the homotopy groups $\pi_i({\rm{Rat}}_d)$ in terms of homotopy groups of spheres $S^2$ and $S^3$. It is a standard fact that among the latter groups only $\pi_2(S^2)$, $\pi_3(S^2)$ and $\pi_3(S^3)$ are infinite. Thus $\pi_i({\rm{Rat}}_d)\otimes\Bbb{Q}=0$ for any 
$i\in\{4,\dots,d-1\}$. Combining this with \eqref{auxiliary1'}, we observe that the only task left is to verify the theorem for $\pi_3(\mathcal{M}_d^{\rm{post}})\otimes\Bbb{Q}$ 
and $\pi_4(\mathcal{M}_d^{\rm{post}})\otimes\Bbb{Q}$.
Tensoring with $\Bbb{Q}$, we have the following exact sequence of $\Bbb{Q}$-vector spaces:
\begin{equation}\label{long exact sequence}
\begin{split}
&0=\pi_4({\rm{PSL}}_2(\Bbb{C}))\otimes\Bbb{Q}\rightarrow\pi_4({\rm{Rat}}_d)\otimes\Bbb{Q}\rightarrow\pi_4(\mathcal{M}_d^{\rm{post}})\otimes\Bbb{Q}
\rightarrow\pi_3({\rm{PSL}}_2(\Bbb{C}))\otimes\Bbb{Q}\\
&\rightarrow\pi_3({\rm{Rat}}_d)\otimes\Bbb{Q}\rightarrow\pi_3(\mathcal{M}_d^{\rm{post}})\otimes\Bbb{Q}\rightarrow\pi_2({\rm{PSL}}_2(\Bbb{C}))\otimes\Bbb{Q}=0.
\end{split}
\end{equation}
The homomorphism 
$$\pi_3({\rm{SL}}_2(\Bbb{C}))\otimes\Bbb{Q}\cong\pi_3({\rm{PSL}}_2(\Bbb{C}))\otimes\Bbb{Q}\rightarrow
\pi_3({\rm{Rat}}_d)\otimes\Bbb{Q}\cong\pi_3(\widetilde{\rm{Rat}}_d)\otimes\Bbb{Q}$$ 
above is induced by the inclusion of an orbit of the post-composition action. This is an isomorphism by Lemma \ref{main lemma variant}.  We conclude that
$\pi_4({\rm{Rat}}_d)\otimes\Bbb{Q}\cong\pi_4(\mathcal{M}_d^{\rm{post}})\otimes\Bbb{Q}$
and $\pi_3(\mathcal{M}_d^{\rm{post}})\otimes\Bbb{Q}=0$.
\end{proof}

\begin{lemma}\label{lemma4}
Given $d>3$, one has 
$$
b_i(\mathcal{R}_d)=
\begin{cases}
0 & 1\leq i< d\text{, and}\\
{\rm{rank}}\,\pi_i({\rm{Rat}}_d) & d\leq i\leq 2d-2.
\end{cases}
$$
\end{lemma}

\begin{proof}
By virtue of Lemma \ref{lemma1}, the homotopy groups $\pi_i(\mathcal{M}_d^{\rm{post}})$ are torsion for $1\leq i\leq d-1$. Thus, by the rational Hurewicz theorem (cf. \cite{MR2055050}), 
the $i^{\rm{th}}$ rational homology and rational homotopy groups of the universal cover of
$\mathcal{M}_d^{\rm{post}}$ coincide whenever $0\leq i\leq 2d-2$. Lemma \ref{lemma1} yields the latter groups; and a universal covering space for $\mathcal{M}_d^{\rm{post}}$ was exhibited in Lemma \ref{lemma3}. Consequently:
\begin{equation}\label{auxiliary10}
H_i\left(\widetilde{\rm{Rat}}_d\big/{\rm{SL}}_2(\Bbb{C});\Bbb{Q}\right)
\cong
\begin{cases}
0 & 1\leq i< d\textit{, and}\\
\pi_i({\rm{Rat}}_d)\otimes\Bbb{Q} & d\leq i\leq 2d-2.
\end{cases}
\end{equation}
We claim that the quotient space $\widetilde{\rm{Rat}}_d\big/{\rm{SL}}_2(\Bbb{C})$ is weakly equivalent to the resultant=1 hypersurface $\mathcal{R}_d$. Identity \eqref{auxiliary10} would then finish the proof. The free action \eqref{action2 post-variant} of ${\rm{SL}}_2(\Bbb{C})$ on $\widetilde{\rm{Rat}}_d$ changes the leading coefficients (the coefficients of $X^d$) of a pair $\left(P(X,Y),Q(X,Y)\right)\in\widetilde{\rm{Rat}}_d$ of homogeneous polynomials of degree $d$ through the canonical left action of ${\rm{SL}}_2(\Bbb{C})$ on $\Bbb{C}^2-\{(0,0)\}$. The latter action is transitive and hence one can represent any element of $\widetilde{\rm{Rat}}_d\big/{\rm{SL}}_2(\Bbb{C})$ by a pair of monic polynomials. However, from its definition, $\mathcal{R}_d$ is the fiber above $(1,1)$ of the map 
$\widetilde{\rm{Rat}}_d\rightarrow\Bbb{C}^2-\{(0,0)\}$ which sends each pair of polynomials to the corresponding vector of leading coefficients; see \eqref{fibration3}. Consequently, $\widetilde{\rm{Rat}}_d\big/{\rm{SL}}_2(\Bbb{C})$ can be identified with $\mathcal{R}_d$ modulo the post-composition action of the subgroup
$$
\left\{\begin{bmatrix}
1-s&s\\
-s&1+s
\end{bmatrix}\,\Bigg|\,s\in\Bbb{C}\right\}\cong(\Bbb{C},+)
$$
which is the stabilizer of $(1,1)$ in the action ${\rm{SL}}_2(\Bbb{C})\curvearrowright\Bbb{C}^2-\{(0,0)\}$. The quotient map 
$$
\mathcal{R}_d\rightarrow \mathcal{R}_d/\Bbb{C}\cong\widetilde{\rm{Rat}}_d\big/{\rm{SL}}_2(\Bbb{C}) 
$$
for this free $\Bbb{C}$-action is a weak homotopy equivalence as it is a fibration with contractible fibers. 
\end{proof}

This series of lemmas culminates in the following:
\begin{proof}[Proof of Theorem \ref{main 3}]
Let $d$ be larger than three. In view of Lemma \ref{lemma4} and the fact that by Theorem \ref{Rat homotopy} the groups $\pi_i({\rm{Rat}}_d)$ are all torsion when $i<d$ except $\pi_3({\rm{Rat}}_d)$ which is of rank one, we have: 
$$
{\rm{rank}}\,\pi_i({\rm{Rat}}_d)=
\begin{cases}
1 & i=3,\\
0 &  1\leq i<d, i\neq 3,\textit{ and}\\
b_i(\mathcal{R}_d) & d\leq i\leq 2d-2.
\end{cases}
$$
On the other hand,  \cite{MR1365252} establishes that 
\begin{equation}\label{comparison based-unbased}
\forall i,d\geq 3: \pi_i({\rm{Rat}}_d)\cong\pi_i({\rm{Rat}}^*_d)\oplus\pi_i(S^3);    
\end{equation}
so the ranks of $\pi_i({\rm{Rat}}_d)$ and 
$\pi_i({\rm{Rat}}^*_d)$ agree whenever $i>3$. Therefore, one only needs to calculate the Betti numbers $b_i(\mathcal{R}_d)$ for $d\leq i\leq 2d-2$. The $\ell$-adic cohomology groups of the resultant=1 hypersurface $\mathcal{R}_d$ have been computed in \cite[Theorem 1.1]{MR3652084} for any $d\geq 1$:
\begin{equation}\label{l-adic Betti}
\forall i>0: \dim_{\Bbb{Q}_\ell}H^i_{et}(\mathcal{R}_{d/\bar{\Bbb{F}}_p};\Bbb{Q}_\ell)=
\begin{cases}
\phi\left(\frac{d}{d-\frac{i}{2}}\right)& \textit{if } 2|i \textit{ and } d-\frac{i}{2}\big|d,\textit{ and}\\
0 &\textit{otherwise},
\end{cases}
\end{equation}
which holds for all but finitely many primes $p$.\footnote{Observe that this also indicates that the \'etale cohomology of $\mathcal{R}_d$ is trivial in positive dimensions less than $d$ 
which is in agreement with the rational homology calculation of Lemma \ref{lemma4}.} 
Notice that $\mathcal{R}_d$ (as a scheme over $\Bbb{Z}$) is smooth: it appears as the fiber of the differentiable fibration $\widetilde{\rm{Rat}}_d\rightarrow\Bbb{C}^2-\{(0,0)\}$ from \eqref{fibration3} (cf. \cite[Remark 1.2]{MR3652084}). Consequently, the Artin comparison theorem can be invoked (see \cite{MR0232775} or \cite[Theorem 2.1]{MR3652084} for the version we are using here) to obtain:
$$
H^i_{et}(\mathcal{R}_{d/\bar{\Bbb{F}}_p};\Bbb{Q}_\ell)\otimes_{\Bbb{Q}_\ell}\Bbb{C}\cong H^i(\mathcal{R}_d;\Bbb{C})
$$
for all but finitely many primes $p$. Combining this all together and letting $j:=d-\frac{i}{2}$, we find for all $d\leq i\leq 2d-2$ that:

$$
{\rm{rank}}\,\pi_i({\rm{Rat}}^*_d)={\rm{rank}}\,\pi_i({\rm{Rat}}_d)=b_i(\mathcal{R}_d)=
\begin{cases}
\phi\left(\frac{d}{j}\right)&\textit{ if } i=2d-2j \textit{ where }1\leq j<d \textit{ divides } d\textit{, and}\\
0 & \textit{otherwise,}\\
\end{cases}
$$
\normalsize
which concludes the proof.
\end{proof}

\begin{remark}
Before moving ahead, we address the rational homotopy groups $\pi_i({\rm{Rat}}^*_d)\otimes\Bbb{Q}$ and $\pi_i({\rm{Rat}}_d)\otimes\Bbb{Q}$ for degrees $d=2,3$ which are absent from Theorem \ref{main 3}. All rational homotopy groups can readily be calculated in the case of quadratic maps: the universal cover of ${\rm{Rat}}_2$ is of the homotopy type of $S^2\times S^3$ while that of the parameter space ${\rm{Rat}}^*_2$ of based quadratic maps is homotopy equivalent to $S^2$; see \cite[Theorem 2.1]{MR1246482} or \cite[Theorem 3]{MR1365252}.\\
\indent
For $d=3$, the Hurewicz homomorphism still yields an isomorphism $\pi_3(\widetilde{\rm{Rat}}_3)\otimes\Bbb{Q}\stackrel{\sim}{\rightarrow}H_3(\widetilde{\rm{Rat}}_3;\Bbb{Q})$ since the second homotopy group of the simply connected space $\widetilde{\rm{Rat}}_3$ is torsion by Theorem \ref{Rat homotopy}; but the theorem does not shed light on $\pi_3(\widetilde{\rm{Rat}}_3)$. Instead of a commutative diagram such as \eqref{diagram} with isomorphisms of one-dimensional $\Bbb{Q}$-vector spaces as rows, we have the following:
\begin{equation*}
\xymatrixcolsep{5pc}\xymatrix{\pi_3({\rm{SL}}_2(\Bbb{C}))\otimes\Bbb{Q}\,\ar[d]^{\sim}_{\text{rational Hurewicz}}\ar@{^{(}->}[r] 
& \pi_3(\widetilde{\rm{Rat}}_3)\otimes\Bbb{Q} \ar[d]^{\text{rational Hurewicz}}_{\sim}\\
H_3({\rm{SL}}_2(\Bbb{C});\Bbb{Q})\, \ar[d]^{\sim} \ar@{^{(}->}[r]& H_3(\widetilde{\rm{Rat}}_3;\Bbb{Q}) \ar@{->>}[d]\\
H_3({\rm{PSL}}_2(\Bbb{C});\Bbb{Q})\ar[r]^{\sim}& H_3({\rm{Rat}}_3;\Bbb{Q})\ar@/_/[u]_{\text{transfer}}.
}
\end{equation*} 
Here, recall that the orbit inclusion gives rise to an isomorphism between the third homology groups of ${\rm{PSL}}_2(\Bbb{C})$ and ${\rm{Rat}}_d$ even for $d=2,3$; see the proof of Theorem \ref{main variant}.
Hence, in view of the long exact sequence \eqref{long exact sequence}, when $d=3$ the only modification necessary in Lemma \ref{lemma4}  is to replace   ${\rm{rank}}\,\pi_3({\rm{Rat}}_3)$ with
${\rm{rank}}\,\pi_3({\rm{Rat}}_3)-1$. Again, by a combination of the Artin comparison theorem and the description \eqref{l-adic Betti} of $\ell$-adic Betti numbers, one obtains the Betti numbers $b_3(\mathcal{R}_3)=0$ and $b_4(\mathcal{R}_3)=\phi(3)=2$. Therefore, 
$$
{\rm{rank}}\,\pi_3({\rm{Rat}}_3)=1, \quad {\rm{rank}}\,\pi_4({\rm{Rat}}_3)=2;
$$
and then by \eqref{comparison based-unbased}
$$
{\rm{rank}}\,\pi_3({\rm{Rat}}^*_3)=0, \quad {\rm{rank}}\,\pi_4({\rm{Rat}}^*_3)=2.
$$
\end{remark}


We conclude with a proof of Theorem \ref{not nilpotent} which follows from our computations above. 
\begin{proof}[Proof of Theorem \ref{not nilpotent}]
We first show that $\mathcal{M}_d^{\rm{post}}$ is not nilpotent. Seeking a contradiction, suppose the action of $\pi_1(\mathcal{M}_d^{\rm{post}})\cong\Bbb{Z}/d\Bbb{Z}$ on the homotopy groups of its universal cover $\widetilde{\rm{Rat}}_d\big/{\rm{SL}}_2(\Bbb{C})$ are nilpotent. As observed before, the Hurewicz map 
$$\pi_i\left(\widetilde{\rm{Rat}}_d\big/{\rm{SL}}_2(\Bbb{C})\right)\otimes\Bbb{Q}\rightarrow H_i\left(\widetilde{\rm{Rat}}_d\big/{\rm{SL}}_2(\Bbb{C});\Bbb{Q}\right)$$
is an isomorphism for $i\leq 2d-2$ and the target vector space is of dimension $b_i(\mathcal{R}_d)$. Hence, the action of $\pi_1(\mathcal{M}_d^{\rm{post}})$ on these homology groups of its universal cover must be nilpotent too. When $i=2d-2$, the group  $H_{2d-2}\left(\widetilde{\rm{Rat}}_d\big/{\rm{SL}}_2(\Bbb{C});\Bbb{Q}\right)$ is of dimension $b_{2d-2}(\mathcal{R}_d)=\phi(d)$ (cf. \eqref{l-adic Betti}) and thus non-trivial. By definition, any non-trivial nilpotent representation admits a non-trivial fixed vector. But here, such a vector amounts to a non-trivial class in  $H_{2d-2}(\mathcal{M}_d^{\rm{post}};\Bbb{Q})$. This contradicts Theorem \ref{main variant}. \\
\indent
We now show that ${\rm{Rat}}_d$ is not nilpotent. The $\Bbb{Q}$-linear map $$\pi_{2d-2}({\rm{Rat}}_d)\otimes\Bbb{Q}\rightarrow\pi_{2d-2}(\mathcal{M}_d^{\rm{post}})\otimes\Bbb{Q}$$
is equivariant with respect to the epimorphism $\pi_1({\rm{Rat}}_d)\twoheadrightarrow\pi_1(\mathcal{M}_d^{\rm{post}})$ and the corresponding group actions. It is also an isomorphism (cf. \eqref{auxiliary1'}). Hence, the action of $\pi_1({\rm{Rat}}_d)$ on $\pi_{2d-2}({\rm{Rat}}_d)\otimes\Bbb{Q}$ cannot be nilpotent since the action of $\pi_1(\mathcal{M}_d^{\rm{post}})$ on $$\pi_{2d-2}(\mathcal{M}_d^{\rm{post}})\otimes\Bbb{Q}\cong\pi_{2d-2}\left(\widetilde{\rm{Rat}}_d\big/{\rm{SL}}_2(\Bbb{C})\right)\otimes\Bbb{Q}$$
is not nilpotent. Repeating the argument, the failure of nilpotency for ${\rm{Rat}}^*_d$ follows from the same property for ${\rm{Rat}}_d$ by considering the $\Bbb{Q}$-linear isomorphism 
$$\pi_{2d-2}({\rm{Rat}}^*_d)\otimes\Bbb{Q}\stackrel{\sim}{\rightarrow}\pi_{2d-2}({\rm{Rat}}_d)\otimes\Bbb{Q}$$ provided by \eqref{fibration1} which is equivariant with respect to $\pi_1({\rm{Rat}}^*_d)\twoheadrightarrow\pi_1({\rm{Rat}}_d)$.
\end{proof}

\bibliography{bib}
\bibliographystyle{alpha}

\end{document}